\documentclass[a4paper]{scrartcl}
\usepackage[utf8]{inputenc}
\usepackage[english]{babel}
\usepackage{microtype}
\usepackage{xcolor}
\usepackage{graphicx}
\usepackage{hyperref}
\usepackage[T1]{fontenc}
\usepackage{lmodern}
\usepackage{rotating}
\usepackage{tikz}
\usepackage{tikz-cd}
\usepackage[backend=biber,sortlocale=de_DE,url=false,doi=false,isbn=false,giveninits=true,style=alphabetic,maxbibnames=1000,date=year,sorting=nyt]{biblatex}
\bibliography{../../../jabref.bib}
\DeclareFieldFormat*{title}{\mkbibemph{#1\isdot}}
\DeclareFieldFormat*{journaltitle}{#1\isdot}
\renewbibmacro{in:}{}
\renewbibmacro{issue+date}{\printtext{\printfield{issue}\hspace{0pt}\!\!\!\phantom{a}\hspace{0pt}\relax\hspace{0pt}(\printdate)}}%allow line break between issue and date

\usepackage{amsmath}
\usepackage{amsfonts}
\usepackage{amssymb}
\usepackage{amsthm}
\usepackage{mathtools}
\usepackage{stmaryrd}
\usepackage{enumerate}
\usepackage{verbatim}
\usepackage{dsfont}
\usepackage{mathabx}
\usepackage{faktor}
\DeclareMathOperator\defect{def}
\DeclareMathOperator\Gal{Gal}

\DeclareMathOperator\conv{conv}
\DeclareMathOperator\QB{QB}
\DeclareMathOperator\wt{wt}

\DeclareMathOperator\inv{inv}

\DeclareMathOperator\avg{avg}
\DeclareMathOperator\supp{supp}

\DeclareMathOperator\LP{LP}
\DeclareMathOperator\cl{cl}

\DeclareMathOperator\Stab{Stab}
\DeclareMathOperator\rk{rk}

\def\GL{{\mathrm{GL}}}

\def\dom{{\mathrm{dom}}}

\hypersetup{colorlinks=false, linkbordercolor={white}, hidelinks, pdfauthor={Felix Schremmer},pdftitle={Generic Newton points and cordial elements}}
\usepackage{csquotes}
\author{Felix Schremmer}\date{\today}
\title{Generic Newton points and cordial elements}
\numberwithin{equation}{section}
\newtheorem{theorem}[equation]{Theorem}
\newtheorem{proposition}[equation]{Proposition}
\newtheorem{lemma}[equation]{Lemma}
\newtheorem{corollary}[equation]{Corollary}

\theoremstyle{definition}
\newtheorem{definition}[equation]{Definition}

\theoremstyle{remark}
\newtheorem{example}[equation]{Example}
\newtheorem{remark}[equation]{Remark}
\def\abs#1{{\left\lvert{#1}\right\rvert}}
\def\doubleparen#1{{(\!({#1})\!)}}
\def\doublebrack#1{{[\![{#1}]\!]}}

\let\oldqedsymbol\qedsymbol
\def\qedaddendum{}
\def\qedsymbol{\oldqedsymbol\qedaddendum}

\def\af{{\mathrm{af}}}
\def\rightqed{\pushQED{\qed}\qedhere\popQED}
\def\presig{\prescript\sigma{}}
\newif\ifthesis
\thesisfalse

\def\weightEstimateShort{L\ref{lem:weightEstimate}}

\begin{document}
%\maketitle
% !TeX spellcheck = en_GB
\maketitle
\begin{abstract}
We describe the generic $\sigma$-conjugacy classes associated with the Iwahori-Bruhat decomposition of a reductive group. As an application, we classify cordial elements as introduced by Mili\'cevi\'c-Viehmann.
\end{abstract}
\section{Introduction}
Affine Deligne-Lusztig varieties play a central role in the study of Shimura varieties, especially the special fibre of integral models. While affine Deligne-Lusztig varieties in the affine flag variety have been studied with great interest in the past, fully satisfactory results only have been achieved under a certain regularity assumption that is typically not satisfied for the examples that come from Shimura varieties. In this paper, we explain how to systematically overcome this regularity condition, allowing us to fully describe the generic Newton points and classify the elements with the most desirable properties.

Let $G$ be a reductive group over a local field $F$. Denote the completion of the maximal unramified extension by $\breve F$ and the Frobenius by $\sigma\in \Gal(\breve F/F)$. There are two important decompositions of the topological space $G(\breve F)$.

Two elements $x_1, x_2\in G(\breve F)$ are $\sigma$-conjugate if $x_1 = y^{-1} x_2 \sigma(y)$ for some $y\in G(\breve F)$. We denote the set of $\sigma$-conjugacy classes by $B(G)$. By Kottwitz \cite{Kottwitz1985, Kottwitz1997}, we know that the $\sigma$-conjugacy class $[g]\in B(G)$ of an element $g\in G(\breve F)$ is uniquely determined by two invariants, namely the Kottwitz point $\kappa(g)\in \pi_1(G)_{\Gamma}$ and the dominant Newton point $\nu(g)\in X_\ast(T)_{\Gamma_0}\otimes\mathbb Q$. Here, $X_\ast(T)$ is the cocharacter group of a maximal torus, $\pi_1(G) = X_\ast(T)/\mathbb Z\Phi^\vee$ its quotient by the root lattice (i.e.\ the Borovoi fundamental group), $\Gamma = \Gal(\overline F/F)$ the absolute Galois group of $F$ and $\Gamma_0 = \Gal(\overline F/\breve F)$ its inertia subgroup.

The closure of a $\sigma$-conjugacy class in the topological space $G(\breve F)$ is a union of $\sigma$-conjugacy classes, defining a partial order on $B(G)$. This order has an accessible combinatorial description in terms of Kottwitz and Newton points, due to Rapoport-Richartz, Viehmann and He, \cite{Rapoport1996, Viehmann2013, He2016}. From Chai \cite{Chai2000}, we know a full description of the set of Kottwitz and Newton points, yielding an explicit description of $B(G)$.

The Iwahori-Bruhat decomposition expresses $G(\breve F)$ as
\begin{align*}
G(\breve F) = \bigsqcup_{x\in \widetilde W} IxI,
\end{align*}
where $I\subseteq G(\breve F)$ denotes a $\sigma$-stable Iwahori subgroup and $\widetilde W$ is the extended affine Weyl group. This decomposition and its applications to the Bruhat-Tits building \cite[Section~4]{Bruhat1972} also have been studied intensively and are well-understood. The extended affine Weyl group $\widetilde W$ is the semi-direct product of the finite Weyl group $W$ and the $\Gamma_0$-coinvariants of the cocharacter group $X_\ast(T)_{\Gamma_0}$.

We are interested in the intersections $IxI\cap [b]$ for $x\in \widetilde W$ and $[b]\in B(G)$, known as \emph{Newton strata}. It is an important open question which Newton strata are non-empty, i.e.\ to describe the set
\begin{align*}
B(G)_x := \{[b]\in B(G)\mid IxI\cap [b]\neq\emptyset\}.
\end{align*}
Related to these intersections are the \emph{affine Deligne-Lusztig varieties} (cf.\ \cite{Rapoport2002}), defined by
\begin{align*}
X_x(b)(\overline{\mathbb F_q}) =\{g\in G(\breve F)/I\mid g^{-1} b\sigma(g)\in IxI\}.
\end{align*}
The dimension and the question of equi-dimensionality of $X_x(b)$ have been intensively studied in the past, yet both problems remain largely open \cite{Goertz2006, Goertz2010, Goertz2010b, He2014, Milicevic2019}. Affine Deligne-Lusztig varieties for certain groups of small rank have been studied explicitly \cite{Reuman2002, Beazley2009, Yang2014}.

Affine Deligne-Lusztig varieties have been introduced by Rapoport~\cite{Rapoport2002} to define Rapoport-Zink moduli spaces, which play an important role for the study of Shimura varieties.

The construction of affine Deligne-Lusztig varieties resembles a classical construction of certain varieties due to Deligne-Lusztig \cite{Deligne1976}. They used the cohomology of these Deligne-Lusztig varieties to describe all complex representations of finite groups of Lie type.

If one replaces the Iwahori subgroup by a hyperspecial subgroup, the resulting affine Deligne-Lusztig varieties in the affine Grassmannian have been well-understood after concentrated effort by many researchers, e.g.\ \cite{Kottwitz2006, Goertz2006, Viehmann2006, Hamacher2015}.

For the affine Deligne-Lusztig varieties considered in this paper, there are a number of important partial results describing their geometry.

It is proved by Görtz-He-Nie \cite{Goertz2015} and Viehmann \cite{Viehmann2021} that $B(G)_x$ always contains a uniquely determined smallest element, which is explicitly described. Moreover, $B(G)_x$ always contains a uniquely determined largest element. This follows from the specialization theorem of Rapoport-Richartz \cite[Theorem~3.6]{Rapoport1996}, as explained by Viehmann~\cite[Proof of Corollary~5.6]{Viehmann2014}. Rapoport-Richartz also prove a version of \emph{Mazur's inequality}, which states that for $[b]\in B(G)_x$ with $x=w\varepsilon^\mu$, we must have an identity of Kottwitz points $\kappa(b) = \kappa(x)$ and the inequality $\nu(b)\leq \mu^{\dom}\in X_\ast(T)_{\Gamma_0}\otimes\mathbb Q$.

While the dimension $\dim X_x(b)$ is difficult to compute, the \emph{virtual dimension} $d_x(b)$ introduced by He \cite{He2014} is easy to evaluate and always an upper bound for $\dim X_x(b)$. Moreover, we have $\dim X_x(b) = d_x(b)$ for a number of cases, but not always. See for example \cite{He2014, Milicevic2020, He2021a}, affirming conjectures of Reuman and others \cite{Reuman2002, Goertz2006}. The virtual dimension is defined as
\begin{align*}
d_x(b) = \frac 12\left(\ell(x) + \ell(\eta_\sigma(x)) - \langle \nu(b),2\rho\rangle-\defect(b)\right).
\end{align*}
Here, $\ell(x)$ denotes the length of $x$ in $\widetilde W$, as explained in Section~\ref{sec:notation}. By $\eta_\sigma(x)$, we denote a certain element in the finite Weyl group associated with $x$, as explained in Section~\ref{sec:root-functionals}. These two terms only depend on the element $x\in \widetilde W$.

The \emph{defect} of a $\sigma$-conjugacy class is a non-negative integer that is bounded by the rank of the root system. We will focus on this invariant in Section~\ref{sec:defect}.

While the results using the virtual dimension are promising, they have important shortcomings. First, it is often assumed that $x$ must lie in a shrunken Weyl chamber, a regularity condition typically not satisfied for examples coming from Shimura varieties. Second, the properties of the virtual dimension are, in general, a lot simpler than those of the actual dimension $\dim X_x(b)$. One can find plenty of examples where the virtual dimension fails to capture the delicate interplay between the elements $x\in \widetilde W$ and $[b]\in B(G)$.

The uniquely determined largest element of $B(G)_x$ is called \emph{generic $\sigma$-conjugacy class} of $x$ and denoted  $[b_x]$. It is the unique $\sigma$-conjugacy class such that $[b_x]\cap IxI$ is dense in $IxI$. The Kottwitz point of $[b_x]$ coincides with the Kottwitz point of $x$, which is easy to compute. The calculation of its Newton point, i.e.\ the \emph{generic Newton point} of $x$, is an important open problem. Our first main result fully solves it, generalizing earlier partial results \cite{Milicevic2021, He2021a, Sadhukhan2021, He2021c}.

In order to write down a concise formula, we have to introduce some important invariants associated with elements in the affine Weyl group (Section~\ref{sec:root-functionals}), elements of the finite Weyl group (Section~\ref{sec:quantumBruhatGraph}) and $\sigma$-conjugacy classes of $G$ (Section~\ref{sec:defect}). These invariants may seem technical at first glance, but each of them has been of fundamental importance in the past, so their relevance for the study of generic Newton points should not surprise.

The first invariant is the set of length positive elements, which we newly introduce in Section~\ref{sec:root-functionals}. It associates to each element $x = w\varepsilon^\mu\in \widetilde W$ a subset $\LP(x)$ of the finite Weyl group. To each $x\in\widetilde W$, we can associate a corresponding alcove in the Bruhat-Tits building, and describe the Weyl chamber containing that alcove by a finite Weyl group element $c\in W$. If we write $x=w\varepsilon^\mu$, then $w^{-1}c\in \LP(x)$. However, the map $\widetilde W\rightarrow W, x\mapsto c$ is somewhat non-canonical, for it depends on the choice of base point in the Bruhat-Tits building. One may consider certain automorphisms of the group $G$ that preserve the Iwahori subgroup $I$, hence induce automorphisms of the Bruhat-Tits building that preserve the base alcove, but not the base point. These automorphisms do not send Weyl chambers to Weyl chambers, explaining why the map associating an element in $\widetilde W$ to its Weyl chamber is ill-behaved in general (see also Example~\ref{ex:usualLPelement} below).

One can think of the set $w\LP(x)$ to encode all Weyl chambers to which $x\in\widetilde W$ is reasonably close. The way this set is defined, it is naturally equivariant under automorphisms as described above. We show that $x$ lies in a shrunken Weyl chamber if and only if $\LP(x)$ consists of only one element, which explains why studying only one particular element in it (i.e.\ the Weyl chamber) is sufficient only for these shrunken cases.

While we are the first to assign a name and a symbol to the concept of length positivity, the concept itself (more or less disguised) plays an important role in a number of different contexts. We mention the Bruhat order and Demazure products of affine Weyl groups, especially the admissible sets relevant to Shimura varieties \cite{Schremmer2022_bruhat}, fundamental elements and $P$-alcoves, especially the non-emptiness of affine Deligne-Lusztig varieties for basic $\sigma$-conjugacy classes \cite{Goertz2015, Lim2023} as well as Kazhdan-Lusztig cells and minimal automata for affine Coxeter groups \cite{Shi1986, Ericksson1994}.

The second invariant needed to state our main results is the weight function of the quantum Bruhat graph. This finite graph was introduced by Brenti-Fomin-Postnikov \cite{Brenti1998} in order to analyse the quantum Chevalley-Monk formula. It since has been studied regularly in the literature on quantum cohomology, e.g.\ in \cite{Postnikov2005}. The connection between the quantum Bruhat graph and the Bruhat order of the affine Weyl group was discovered by Lam-Shimozono \cite{Lam2010}. This allowed Mili\'cevi\'c \cite{Milicevic2021} to use the quantum Bruhat graph in conjunction with a result of Viehmann \cite{Viehmann2014} to obtain a description of generic Newton points in special cases. Since this initial article, the quantum Bruhat graph has been widely used to study affine Deligne-Lusztig varieties \cite{Milicevic2020, He2021d, He2021c, Sadhukhan2021, Sadhukhan2022}.

The quantum Bruhat graph is a directed graph with vertex set $W$, the finite Weyl group. By considering weights of shortest paths, we associate a value $\wt(u\Rightarrow v)\in X_\ast(T)_{\Gamma_0}$ for arbitrary elements $u,v\in W$. It should be clear that some invariant of this form is needed, since the generic $\sigma$-conjugacy class of an element $x = w\varepsilon^\mu$ should certainly depend on the finite Weyl group element $w$ and the Weyl chamber containing $x$, or more naturally the set $\LP(x)$, which gives a second element of the finite Weyl group.

The final invariant needed to state our first main result is the $\lambda$-invariant $\lambda(b)\in X_\ast(T)_{\Gamma}$ of a $\sigma$-conjugacy class $[b]\in B(G)$. It was introduced by Hamacher-Viehmann \cite{Hamacher2018} in order to study irreducible components of affine Deligne-Lusztig varieties in the affine Grassmannian. It generalizes the \enquote{best integral approximation} of the Newton point $\nu(b)$ studied e.g.\ by Kottwitz \cite{Kottwitz2006} for split groups $G$ whose derived subgroup is simply connected. The $\lambda$-invariant and the Newton point are closely related, which we will discuss in Sections \ref{sec:parabolic-averages} and \ref{sec:defect}.

For simplicity, we state our main results only for the case of quasi-split groups, referring the reader interested in non quasi-split groups to Section~\ref{sec:gnpArbitraryGroups}.
\begin{theorem}[Cf.\ Theorem~\ref{thm:genericGKP} and Corollary~\ref{cor:genericGKPMinDistance}]\label{thm:introGNP}
Let $x = w\varepsilon^\mu\in \widetilde W$. Then the $\lambda$-invariant of the generic $\sigma$-conjugacy class $[b_x]\in B(G)_x$ is given by
\begin{align*}
\lambda(b_x) = \max_{v\in W} v^{-1}\mu -\wt(v\Rightarrow\sigma(wv))\in X_\ast(T)_{\Gamma}.
\end{align*}
It can be calculated more explicitly as follows: Pick an element $v\in \LP(x)$ such that the distance $d(v\Rightarrow\sigma(wv))$ in the quantum Bruhat graph becomes minimal. Then
\begin{align*}
\lambda(b_x) = v^{-1}\mu - \wt(v\Rightarrow\sigma(wv))\in X_\ast(T)_{\Gamma}.
\end{align*}
The generic Newton point is given by
\begin{align*}
\nu(b_x) = \max_J\pi_J(\lambda(b_x)),
\end{align*}
where the maximum is taken over all Frobenius-stable sets of simple roots $J$ and the function $\pi_J$ is the corresponding averaging function, cf.\ \cite[Definition~3.2]{Chai2000} or Section~\ref{sec:parabolic-averages} below.
\end{theorem}
This theorem may be seen as a refinement of the aforementioned Mazur inequality, as it gives a sharp upper bound for $\{\nu(b)\mid [b]\in B(G)_x\}$. The concise formula for the $\lambda$-invariant is useful for our second main result.

If the dimension coincides with the virtual dimension for the generic $\sigma$-conjugacy class, i.e.\ $\dim X_x(b_x) = d_x(b_x)$, the element $x$ is called \emph{cordial} following Mili\'cevi\'c-Viehmann \cite{Milicevic2020}. They prove in \cite[Corollary~3.17, Theorem~1.1]{Milicevic2020} that cordial elements satisfy the most desirable properties. In particular, the set $B(G)_x$ is explicitly described as a closed interval in $B(G)$, and for each $[b]\in B(G)_x$, the affine Deligne-Lusztig variety $X_x(b)$ is equi-dimensional of dimension $d_x(b)$. Using our result on generic Newton points, we are able to fully classify the cordial elements in $\widetilde W$.
\begin{theorem}[Cf.\ Corollary~\ref{cor:cordial}]
Let $x=w\varepsilon^\mu\in \widetilde W$ and $v\in W$ such that $v^{-1}\mu$ is dominant and $v$ has minimal length. Equivalently, this is the element such that the alcove of $x$ lies in the Weyl chamber defined by $wv\in W$. Then $x$ is cordial if and only if the following two conditions are both satisfied:
\begin{enumerate}[(a)]
\item The element $v$ satisfies the minimal distance condition of Theorem~\ref{thm:introGNP}, i.e.\
\begin{align*}
d(v\Rightarrow\sigma(wv))\leq d(v'\Rightarrow \sigma(wv'))
\end{align*}
for every $v'\in\LP(x)$.
\item We have $d(v\Rightarrow\sigma(wv)) = \ell(v^{-1}\sigma(wv))$.
\end{enumerate}
\end{theorem}
Condition (a) of the above theorem can be thought of as a regularity condition on the affine Weyl group elements, since it is automatic for elements in a shrunken Weyl chamber. Condition (b) is purely about combinatorics of the finite Weyl group. It is an extremality condition, since the inequality $d(u\Rightarrow v)\leq \ell(u^{-1} v)$ is true for all $u,v\in W$.

The theory of cordial elements has been used by He \cite{He2021a} to compute the dimensions of many affine Deligne-Lusztig varieties, even for non-cordial elements $x\in \widetilde W$.

In order to prove our main results, we introduce new methods and refine existing ones. We introduce the language of length functionals in Section~\ref{sec:root-functionals}, which is essential when studying elements $x\in \widetilde W$ without regularity assumptions (e.g.\ minuscule $x$). A few new insights on the quantum Bruhat graph complement the combinatorics needed to prove our results.

As a preparation for the more geometric aspects of our proofs, we review and refine a number of known results on $\sigma$-conjugacy classes in Section~\ref{chap:sigma-conjugation}. Our main results hold true whenever $G$ is connected and reductive. Following Görtz-He-Nie \cite{Goertz2015}, we can prove this via a reduction to the case where $G$ is quasi-split. However, many important foundational results have been proved only under the somewhat stricter assumption that $G$ should be unramified. We show how to generalize these classical results to the quasi-split case, allowing us to prove our main results in this setting (Corollaries \ref{cor:genericGKPMinDistance} and \ref{cor:cordial}). This enables us to conclude them for arbitrary connected reductive groups (Theorem~\ref{thm:gnpGeneralGroups} and Proposition~\ref{prop:cordialGeneralGroups}).

This paper covers parts of the author's PhD thesis.
\subsection{Acknowledgements}
First and foremost, I would like to thank my advisor Eva Viehmann for her constant support throughout my PhD time. I am deeply thankful for her invaluable help in both mathematical and administrative matters.

I would like to thank Paul Hamacher and Xuhua He for inspiring discussions.

The author was partially supported by the ERC Consolidator Grant 770936: \emph{NewtonStrat}, the German Academic Scholarship Foundation, the Marianne-Plehn-Program and the DFG Collaborative Research Centre 326: \emph{GAUS}.

% !TeX spellcheck = en_GB
\section{The affine root system}\label{chap:affine-root-system}
\subsection{Group-theoretic setup}\label{sec:notation}
We fix a non-archimedian local field $F$ whose completion of the maximal unramified extension will be denoted $L = \breve F$. We write $\mathcal O_F$ and $\mathcal O_L$ for the respective rings of integers. Let $\varepsilon \in F$ be a uniformizer. The Galois group $\Gamma = \Gal(L/F)$ is generated by the Frobenius $\sigma$.

Concretely, this means we have one of the following situations:
\begin{itemize}
\item Mixed characteristic case: $F/\mathbb Q_p$ is a finite extension for some prime $p$. Then $\mathcal O_F$ is the set of integral elements of $F$.
\item Equal characteristic case: $\mathcal O_F$ is a ring of formal power series $\mathbb F_q\doublebrack\varepsilon$, $F = \mathbb F_q\doubleparen\varepsilon$ is its fraction field, $\mathcal O_L = \overline{\mathbb F_q}\doublebrack\varepsilon$ and $L = \overline{\mathbb F_q}\doubleparen\varepsilon$. The Frobenius $\sigma$ acts on $L$ via
\begin{align*}
\sigma\left(\sum a_n \varepsilon^n\right) = \sum a_n^q \varepsilon^n.
\end{align*}
\end{itemize}

We consider a connected and reductive group $G$ over $F$. We construct its associated affine root system and affine Weyl group following Haines-Rapoport \cite{Haines2008} and Tits \cite{Tits1979}.

Fix a maximal $L$-split torus $S\subseteq G_L$ and  write $T$ for its centralizer in $G_L$, so $T$ is a maximal torus of $G_L$. Write $\mathcal A = \mathcal A(G_L,S)$ for the apartment of the Bruhat-Tits building of $G_L$ associated with $S$. We pick a $\sigma$-invariant alcove $\mathfrak a$ in $\mathcal A$. This yields a $\sigma$-stable Iwahori subgroup $I\subset G(L)$.

Denote the normalizer of $T$ in $G$ by $N(T)$. Then the quotient \begin{align*}\widetilde W = N_G(T)(L) / (T(L)\cap I)\end{align*} is called \emph{extended affine Weyl group}, and $W = N_G(T)(L)/T(L)$ is the \emph{(finite) Weyl group}. The Weyl group $W$ is naturally a quotient of $\widetilde W$.

The affine roots as constructed in \cite[Section~1.6]{Tits1979} are denoted $\Phi_\af$. Each of these roots $a\in \Phi_\af$ defines an affine function $a:\mathcal A\rightarrow\mathbb R$. The vector part of this function is denoted $\cl(a) \in V^\ast$, where $V = X_\ast(S)\otimes\mathbb R = X_\ast(T)_{\Gamma_0}\otimes \mathbb R$. Here, $\Gamma_0 = \Gal(\overline L/L)$ is the absolute Galois group of $L$, i.e.\ the inertia group of $\Gamma = \Gal(\overline F/F)$. The set of \emph{(finite) roots} is\footnote{This is different from the root system that \cite{Tits1979} and \cite{Haines2008} denote by $\Phi$; it coincides with the root system called $\Sigma$ in \cite{Haines2008}.} $\Phi := \cl(\Phi_\af)$.

The affine roots in $\Phi_\af$ whose associated hyperplane is adjacent to our fixed alcove $\mathfrak a$ are called \emph{simple affine roots} and denoted $\Delta_\af\subseteq \Phi_\af$.

Writing $W_\af$ for the extended affine Weyl group of the simply connected quotient of $G$, we get a natural $\sigma$-equivariant short exact sequence (cf.\ \cite[Lemma~14]{Haines2008})
\begin{align*}
1\rightarrow W_\af\rightarrow\widetilde W\rightarrow \pi_1(G)_{\Gamma_0}\rightarrow 1.
\end{align*}
Here, $\pi_1(G) := X_\ast(T)/\mathbb Z\Phi^\vee$ denotes the Borovoi fundamental group.

For each $x\in \widetilde W$, we denote by $\ell(x)\in \mathbb Z_{\geq 0}$ the length of a shortest alcove path from $\mathfrak a$ to $x\mathfrak a$. The elements of length zero are denoted $\Omega$. The above short exact sequence yields an isomorphism of $\Omega$ with $\pi_1(G)_{\Gamma_0}$, realizing $\widetilde W$ as semidirect product $\widetilde W = \Omega\ltimes W_\af$.

Each affine root $a\in \Phi_\af$ defines an affine reflection $r_a$ on $\mathcal A$. The group generated by these reflections is naturally isomorphic to $W_\af$ (cf.\ \cite{Haines2008}), so we also write $r_a\in W_\af$ for the corresponding element. We define $S_\af := \{r_a\mid a\in \Delta_\af\}$, called the set of \emph{simple affine reflections}. The pair $(W_\af, S_\af)$ is a Coxeter group with length function $\ell$ as defined above.

We pick a special vertex $\mathfrak x\in \mathcal A$ that is adjacent to $\mathfrak a$. We identify $\mathcal A$ with $V$ via $\mathfrak x\mapsto 0$. This allows us to decompose $\Phi_\af = \Phi\times\mathbb Z$, where $a = (\alpha,k)$ corresponds to the function
\begin{align*}
V\rightarrow \mathbb R, v\mapsto \alpha(v)+k.
\end{align*}
From \cite[Proposition~13]{Haines2008}, we moreover get decompositions $\widetilde W = W\ltimes X_\ast(T)_{\Gamma_0}$ and $W_\af = W\ltimes \mathbb Z\Phi^\vee$. Using this decomposition, we write elements $x\in \widetilde W$ as $x = w\varepsilon^\mu$ with $w\in W$ and $\mu\in X_\ast(T)_{\Gamma_0}$. For $a = (\alpha,k)\in \Phi_\af$, we have $r_a = s_\alpha \varepsilon^{k\alpha^\vee}\in W_\af$, where $s_\alpha\in W$ is the reflection associated with $\alpha$. The natural action of $\widetilde W$ on $\Phi_\af$ can be expressed as
\begin{align*}
(w\varepsilon^\mu)(\alpha,k) = (w\alpha,k-\langle\mu,\alpha\rangle).
\end{align*}
We define the \emph{dominant chamber} $C\subseteq V$ to be the Weyl chamber containing our fixed alcove $\mathfrak a$. This gives a Borel subgroup $B\subseteq G$, and corresponding sets of positive/negative/simple roots $\Phi^+, \Phi^-, \Delta\subseteq \Phi$.

By abuse of notation, we denote by $\Phi^+$ also the indicator function of the set of positive roots, i.e.
\begin{align*}
\Phi^+:\Phi\rightarrow\{0,1\},\quad \alpha\mapsto\begin{cases}1,&\alpha\in \Phi^+,\\
0,&\alpha\in \Phi^-.\end{cases}
\end{align*}
The following easy facts will be used often, usually without further reference:
\begin{lemma}\label{lem:phiPlusFacts}
Let $\alpha\in \Phi$.
\begin{enumerate}[(a)]
\item $\Phi^+(\alpha) + \Phi^+(-\alpha)=1$.
\item If $\beta\in \Phi$ and $k,\ell\geq 1$ are such that $k\alpha+\ell \beta\in \Phi$, we have
\begin{align*}
&0\leq \Phi^+(\alpha)+\Phi^+(\beta)-\Phi^+(k\alpha+\ell\beta)\leq 1.\rightqed
\end{align*}
\end{enumerate}
\end{lemma}
The sets of positive and negative affine roots can be defined as
\begin{align*}
\Phi_\af^+:=&(\Phi^+\times \mathbb Z_{\geq 0})\sqcup (\Phi^-\times \mathbb Z_{\geq 1}) = \{(\alpha,k)\in \Phi_\af\mid k\geq \Phi^+(-\alpha)\},
\\\Phi_\af^- :=&-\Phi_\af^+ = \Phi_\af\setminus \Phi_\af^+= \{(\alpha,k)\in \Phi_\af\mid k< \Phi^+(-\alpha)\}.
\end{align*}
One checks that $\Phi_\af^+$ are precisely the affine roots that are sums of simple affine roots.

Decompose $\Phi = \Phi_1\sqcup\cdots\sqcup \Phi_r$ as a direct sum of irreducible root systems. Each irreducible factor contains a uniquely determined longest root $\theta_i\in \Phi_i^+$. Now the set of simple affine roots is
\begin{align*}
\Delta_\af = \{(\alpha,0)\mid \alpha\in \Delta\}\sqcup\{(-\theta_i,1)\mid i=1,\dotsc,r\}\subset \Phi_\af^+.
\end{align*}

The \emph{Bruhat order} on $W_a$ is the usual Coxeter-theoretic notion. The Bruhat order on $\widetilde W$ can be defined as $\omega x\leq \omega'x'$ iff $\omega = \omega'$ and $x\leq x'$ for $\omega,\omega'\in \Omega$ and $x,x'\in W_a$.

We call an element $\mu\in X_\ast(T)_{\Gamma_0}\otimes \mathbb Q$ \emph{dominant} if $\langle \mu,\alpha\rangle\geq 0$ for all $\alpha\in \Phi^+$. For elements $\mu,\mu'$ in $X_\ast(T)_{\Gamma_0}\otimes \mathbb Q$ (resp.\ $X_\ast(T)_{\Gamma_0}$ or $X_\ast(T)_{\Gamma}$), we write $\mu\leq \mu'$ if the difference $\mu'-\mu$ is a $\mathbb Q_{\geq 0}$-linear combination of positive coroots.

The induced action of $\Gamma_0$ on $\mathcal A, \Phi_\af, \widetilde W, W_\af$ and $W$ is trivial by construction. The Frobenius action on $\mathcal A, X_\ast(T)_{\Gamma_0}, \Phi_\af$ and $\Phi$ will be denoted by $\sigma$. Note that $\sigma$ preserves the set of simple affine roots. The Frobenius action on $W, \widetilde W$ and $W_\af$ will be denoted by $x\mapsto \presig x$. Then the action of $\presig x$ on $X_\ast(T)_{\Gamma_0}$ is the same as the composed action $\sigma\circ x\circ\sigma^{-1}$ ($x\in W$ or $\widetilde W$).

For the most part, we consider the case where $G$ is quasi-split over $F$. This is a convenient assumption that lightens the notational burden significantly. In Section~\ref{sec:gnpArbitraryGroups}, we return to the more general setting of connected reductive $G$ and generalize our main results via a reduction to the quasi-split case.

If $G$ is quasi-split, we may and do choose the vertex $\mathfrak{x}$ to be $\sigma$-invariant. With this choice, the decompositions $\Phi_\af = \Phi\times\mathbb Z$ and $\widetilde W = W\ltimes X_\ast(T)_{\Gamma_0}$ are Frobenius equivariant. This means
\begin{align*}
\forall(\alpha,k)\in \Phi_\af:~&\sigma(\alpha,k) = (\sigma(\alpha),k),\\
\forall w\varepsilon^\mu\in \widetilde W:~&\presig{}\!\left(w\varepsilon^\mu\right) = (\presig w)\varepsilon^{\sigma(\mu)}.
\end{align*}
In particular, $\sigma$ preserves the set of simple roots $\Delta$.

The case where $G$ is unramified has often been studied in the literature. In this case, $S$ is a maximal torus of $G_L$, so $S=T$ and $\Phi$ is the usual root system of $(G,T)$. Each root system $\Phi$ together with a Frobenius action comes from such an unramified group. However, care has to be taken when using results proved for unramified groups in the quasi-split setting, as $X_\ast(T)_{\Gamma_0}$ may have a torsion part if $G$ is not unramified. In particular, the map $X_\ast(T)_{\Gamma_0}\rightarrow X_\ast(T)_{\Gamma_0}\otimes\mathbb R=V\cong \mathcal A$ might fail to be injective.
\subsection{Root functionals}\label{sec:root-functionals}
For every coweight $\mu$, there exists a uniquely determined dominant coweight in the $W$-orbit of $\mu$. In other words, there exists some $w\in W$ such that $\mu(w\alpha)\geq 0$ for all $\alpha\in \Phi^+$.

In this section, we introduce and study certain functions $\varphi:\Phi\rightarrow\mathbb Z$ which are more general than coweights, but still enjoy this property.
\begin{definition}\label{def:rootFunctionals}
\begin{enumerate}[(a)]
\item
A \emph{root functional} is a function $\varphi:\Phi\rightarrow\mathbb Z$ satisfying the following two conditions for all $\alpha,\beta\in \Phi$:
\begin{enumerate}[(1)]
\item $\abs{\varphi(\alpha)+\varphi(-\alpha)}\leq 1$.
\item If $\alpha+\beta\in \Phi$, then
\begin{align*}
\abs{\varphi(\alpha+\beta)-\varphi(\alpha)-\varphi(\beta)}\leq 1.
\end{align*}
\end{enumerate}
\item If $\varphi$ is a root functional, the \emph{dual root functional} $\varphi^\vee$ is defined for $\alpha\in \Phi$ by $\varphi^\vee(\alpha) = -\varphi(-\alpha)$.
\item Let $v\in W$. The set of \emph{inversions} of $v$ with respect to $\varphi$ is
\begin{align*}
\inv_\varphi(v)=\{\alpha\in \Phi^+\mid \varphi(v\alpha)<0\}\cup \{\alpha\in \Phi^-\mid \varphi(v\alpha)>0\}.
\end{align*}
We call $v$ \emph{positive} for $\varphi$ if $\inv_\varphi(v)=\emptyset$. If $\alpha\in \inv_\varphi(v)$, we call $vs_\alpha\in W$ an \emph{adjustment} of $v$ for $\varphi$.
\end{enumerate}
\end{definition}
\begin{lemma}\label{lem:rootFunctionalAdjustment}
Let $\varphi : \Phi\rightarrow\mathbb Z$ be a root functional and $v\in W$ be \emph{not} positive for $\varphi$. If $v'$ is an adjustment of $v$ for $\varphi$, then
\begin{align*}
\#\inv_\varphi(v')<\#\inv_\varphi(v).
\end{align*}
\end{lemma}
\begin{proof}
Let $\alpha\in \inv_\varphi(v)$ with $v' = vs_\alpha$. Up to replacing $(\alpha,\varphi)$ by $(-\alpha,\varphi^\vee)$, we may assume $\alpha\in \Phi^+$, so $\varphi(v\alpha)<0$.
Define
\begin{align*}
I := \{\beta\in \Phi^+\setminus\{\alpha\}\mid s_\alpha(\beta)\in \Phi^-\}.
\end{align*}
We write
\begin{align*}
\#\inv_{\varphi}(v') =& \#\{\beta\in \Phi^+\setminus I\mid \varphi(v'\beta)<0\} + \#\{\beta \in I\mid \varphi(v'\beta)<0\}
\\&+\#\{\beta\in \Phi^-\setminus (-I)\mid \varphi(v'\beta)>0\} + \#\{\beta \in -I\mid \varphi(v'\beta)>0\}
\end{align*}
Note that $\varphi(v'\alpha) = \varphi(-v\alpha) \geq -1-\varphi(v\alpha)\geq 0$ and $s_\alpha(\Phi^+\setminus (I\cup\{\alpha\})) = \Phi^+\setminus(I\cup\{\alpha\})$. Thus
\begin{align*}
\#\{\beta\in \Phi^+\setminus I\mid \varphi(v'\beta)<0\} =& \#\{\beta\in \Phi^+\setminus (I\cup\{\alpha\})\mid \varphi(v s_\alpha\beta)<0\}
\\=&\#\{\beta\in \Phi^+\setminus (I\cup\{\alpha\})\mid \varphi(v \beta)<0\}
\\=&\#\{\beta\in \Phi^+\setminus I\mid \varphi(v \beta)<0\}-1.
\end{align*}
Similarly, we have
\begin{align*}
\#\{\beta\in \Phi^-\setminus (-I)\mid \varphi(v'\beta)>0\}=&\#\{\beta\in \Phi^-\setminus(-I\cup\{-\alpha\})\mid \varphi(v'\beta)>0\}
\\=&\#\{\beta\in \Phi^-\setminus(-I\cup\{-\alpha\})\mid \varphi(v\beta)>0\}
\\\leq &\#\{\beta\in \Phi^-\setminus (-I)\mid \varphi(v \beta)>0\}.
\end{align*}

Therefore, it suffices to prove the following estimates:
\begin{align*}
&\#\{\beta \in I\mid \varphi(v'\beta)<0\} \leq \#\{\beta \in I\mid \varphi(v\beta)<0\},\tag{1}
\\&\#\{\beta \in -I\mid \varphi(v'\beta)>0\} \leq \#\{\beta \in -I\mid \varphi(v\beta)>0\}.\tag{2}
\end{align*}
We only prove (1), as the proof of (2) is similar.

In order to prove (1), we consider the involution $\beta\mapsto -s_\alpha(\beta)$, which acts freely on $I$. Let $o = \{\beta,-s_\alpha(\beta)\}\subseteq I$ be an orbit for this involution. It suffices to show
\begin{align*}
\#\{\beta \in o\mid \varphi(v'\beta)<0\} \leq \#\{\beta \in o\mid \varphi(v\beta)<0\}.\tag{$\ast$}
\end{align*}
In order to prove this, we calculate
\begin{align*}
\#\{\beta \in o\mid \varphi(v'\beta)<0\} =& \#\{\beta \in -s_{\alpha}(o)\mid \varphi(v'\beta)<0\}
\\=&\#\{\beta \in o\mid \varphi(-v\beta)<0\}
\\\leq&\#\{\beta\in o\mid \varphi(v\beta)\geq 0\}
\\=&2-\#\{\beta \in o\mid \varphi(v\beta)<0\}.
\end{align*}
If $\#\{\beta \in o\mid \varphi(v\beta)<0\}\geq 1$, we immediately get $(\ast)$.

Now suppose that $\varphi(v\beta)\geq 0$ for all $\beta\in o$. Fix an element $\beta\in o$ and write \begin{align*}\beta' := -s_\alpha(\beta) = \langle \alpha^\vee,\beta\rangle \alpha-\beta.
\end{align*}
Note that $k\alpha-\beta\in \Phi$ for $k=0,\dotsc,\langle \alpha^\vee,\beta\rangle$. Thus
\begin{align*}
&\abs{\varphi(v\beta') - \langle \alpha^\vee,\beta\rangle \varphi(v\alpha) - \varphi(-v\beta)}
\\\leq&\sum_{k=1}^{\langle \alpha^\vee,\beta\rangle} \abs{\varphi(v(k\alpha-\beta)) - \varphi(v\alpha) - \varphi(v(k-1)\alpha-\beta)}
\\\leq&\langle \alpha^\vee,\beta\rangle.
\end{align*}
In particular, we get
\begin{align*}
\varphi(v\beta') - \varphi(-v\beta) \leq \langle \alpha^\vee,\beta\rangle(1+\varphi(v\alpha)) \leq 0.
\end{align*}
Thus $\varphi(-v\beta)\geq \varphi(v\beta')\geq 0$.

Since $\beta\in o$ was arbitrary, we get $\varphi(v'\beta) = \varphi(-v(-s_\alpha)\beta)\geq 0$ for all $\beta\in o$. This proves $(\ast)$, which finishes the proof of the lemma.
\end{proof}
\begin{corollary}\label{cor:rootFunctionalAdjustments}
If $\varphi: \Phi\rightarrow\mathbb Z$ is a root functional and $v\in W$ is any element, there is a sequence
\begin{align*}
v = v_1,\dotsc,v_k\in W
\end{align*}
such that $v_{i+1}$ is an adjustment for $v_i$ for $\varphi$ (where $i=1,\dotsc,k-1$), and $v_k$ is positive for $\varphi$. In particular, positive elements exist for each root functional.\rightqed
\end{corollary}
The most important root functional for us will be the length functional associated to an element $x\in \widetilde W$, which we introduce now.
\begin{definition}
Let $x = w\varepsilon^\mu\in \widetilde W$ and $\alpha\in \Phi$. We define
\begin{align*}
\ell(x,\alpha) := \langle \mu,\alpha\rangle +\Phi^+(\alpha) - \Phi^+(w\alpha).
\end{align*}
\end{definition}
The absolute value $\abs{\ell(x,\alpha)}$ can be understood as counting affine root hyperplanes between the base alcove and $x\mathfrak a$, while the sign accounts for the orientations (cf.\ Lemma~\ref{lem:lengthFunctionalAsCountingAffineRoots}).
\begin{lemma}
Let $x=w\varepsilon^\mu\in \widetilde W$. Then $\ell(x,\cdot)$ is a root functional. For each $\alpha\in \Phi$, we have
$
\ell(x,\alpha) + \ell(x,-\alpha)=0.
$
\end{lemma}
\begin{proof}
\begin{enumerate}[(1)]
\item We have
\begin{align*}
&\ell(x,\alpha) + \ell(x,-\alpha) \\=& \langle \mu,\alpha\rangle +\Phi^+(\alpha) - \Phi^+(w\alpha) + \langle \mu,-\alpha\rangle + \Phi^+(-\alpha) - \Phi^+(-w\alpha)
\\=&\Phi^+(\alpha) + \Phi^+(-\alpha) - (\Phi^+(w\alpha) + \Phi^+(-w\alpha)) = 1-1=0.
\end{align*}
\item 
Let $\alpha,\beta\in \Phi$ such that $\alpha+\beta\in \Phi$. We know that
\begin{align*}
0\leq \Phi^+(\alpha) + \Phi^+(\beta)-\Phi^+(\alpha+\beta) \leq 1.
\end{align*}
Thus, we obtain
\begin{align*}
&\abs{\ell(x,\alpha+\beta)-\ell(x,\alpha)-\ell(x,\beta)}
\\=&\lvert\underbrace{\Phi^+(\alpha+\beta)-\Phi^+(\alpha)-\Phi^+(\beta)}_{\in \{-1,0\}} \underbrace{- \Phi^+(w(\alpha+\beta)) + \Phi^+(w\alpha)+\Phi^+(w\beta)}_{\in \{0,1\}}\rvert\leq 1.
\end{align*}
\end{enumerate}
This finishes the proof.
\end{proof}
\begin{definition}Let $x\in \widetilde W$ and $v\in W$.
We say that $v$ is \emph{length positive for $x$} and write $v\in \LP(x)$ if $v$ is positive for the length functional $\ell(x,\cdot)$. Explicitly, $v$ is length positive for $x$ if $\ell(x,v\alpha)\geq 0$ for all $\alpha\in \Phi^+$.
\end{definition}
\begin{example}\label{ex:usualLPelement}
Let $x = w\varepsilon^\mu\in \widetilde W$. The $W$-orbit of $\mu$ contains a unique dominant element of $X_\ast(T)_{\Gamma_0}$, and there is a unique $v\in W$ of minimal length such that $v^{-1}\mu$ is dominant. The element $v$ is uniquely determined by the following condition for each positive root $\alpha$:
\begin{align*}
\langle v^{-1}\mu,\alpha\rangle \geq \Phi^+(-v\alpha).
\end{align*}
It follows that
\begin{align*}
\ell(x,v\alpha) = \langle v^{-1}\mu,\alpha\rangle - \Phi^+(-v\alpha) + \Phi^+(-wv\alpha)\geq 0.
\end{align*}
We see that this particular $v$ is length positive. This gives an alternative proof that length positive elements always exist.

With $v$ as above, one checks that the alcove of $x$ in the Bruhat-Tits building lies in the Weyl chamber determined by $wv\in W$.

Recall the definition of the virtual dimension for $x\in \widetilde W$ and $b\in B(G)$.
\begin{align*}
d_x(b) = \frac 12\left(\ell(x) +\ell(\eta_\sigma(x))-\langle\nu(b),2\rho\rangle-\defect(b)\right).
\end{align*}
Here, $2\rho\in X_\ast(T)^{\Gamma}$ denotes the sum of positive roots.
With $v\in W$ constructed as above, we have
\begin{align*}
\eta_\sigma(x) = \prescript{\sigma^{-1}}{}(v) ^{-1}wv\in W.
\end{align*}
Because of the importance of the virtual dimension, the specific $v$ constructed in this example is of particular interest.

However, the construction of this $v\in W$ is not quite natural in terms of $x\in \widetilde W$, e.g.\ in view of certain automorphisms of $\widetilde W$ that preserve dimensions of affine Deligne-Lusztig varieties. These can be automorphisms of the affine Dynkin diagram, e.g.\ coming from conjugation with a length zero element in $\widetilde W$ or the longest element $w_0\in W$. If $G$ is split, then $X_x(b)\cong X_{x^{-1}}(b^{-1})$ for all $x,b\in \widetilde W$.

Studying the group $\GL_3$ for a concrete example, there are three simple affine reflections $s_0, s_1, s_2$ in $\widetilde W$. Each of these satisfies $\ell(s_i)=\dim X_{s_i}(1)=1$. The two simple affine reflections $s_1$ and $s_2$ that come from $W$ also satisfy $\ell(\eta_\sigma(s_1)) = \ell(\eta_\sigma(s_2))=1$, so that
\begin{align*}
d_{s_i}([1]_\sigma) = \frac 12\left(1+1-0-0\right)=1 = \dim X_{s_i}(1),\qquad i=1,2.
\end{align*}
For the remaining affine simple reflection $s_0$, we have $\ell(\eta_\sigma(s_0))=3$. Thus $d_{s_0}(1) = 2 > \dim X_{s_0}(1)$.

We see that $s_1, s_2$ satisfy $\dim X_{s_i}(1)= d_{s_i}(1)$ (so both are cordial), whereas $s_0$ does not have this property. This is problematic insofar as there exists an automorphism of the affine Dynkin diagram sending $s_1$ to $s_0$, hence naturally $X_{s_0}(1) \cong X_{s_1}(1)$. This natural isomorphism is not reflected in the corresponding virtual dimensions, which comes precisely from the term $\ell(\eta_\sigma(x))$.

Searching for a replacement of this specific $v$ that is invariant under such automorphisms, we found the notion of length positive elements. The set of length positive elements is well-behaved under such automorphisms, as it allows the following root-theoretic interpretation.
\end{example}
\begin{lemma}[{cf. \cite[Lemma~3.12]{Lenart2015}}]\label{lem:lengthFunctionalAsCountingAffineRoots}
Let $x = w\varepsilon^\mu \in \widetilde W$ and $\alpha\in \Phi$. Then
\begin{align*}
\#\{k\in\mathbb Z\mid (\alpha,k)\in \Phi^+_\af\text{ and }x(\alpha,k)\in \Phi^-_\af\} = \max(0,\ell(x,\alpha)).
\end{align*}
\end{lemma}
\begin{proof}
We have
\begin{align*}
&\{(\alpha,k)\in \Phi_\af^+ \mid x(\alpha,k) \in \Phi_\af^-\}
\\=& \{(\alpha,k) \in \Phi_\af \mid k\geq \Phi^+(-\alpha)\text{ and }(w\alpha,k-\langle \mu,\alpha\rangle)\in \Phi_\af^-\}
\\=&\{(\alpha,k) \in \Phi_\af \mid k\geq \Phi^+(-\alpha)\text{ and }k-\langle \mu,\alpha\rangle \leq -\Phi^+(w\alpha)\}.
\\\cong&\{k\in \mathbb Z\mid \Phi^+(-\alpha)\leq k \leq \langle \mu,\alpha\rangle -\Phi^+(w\alpha)\}.
\end{align*}
The cardinality of this set is given by
\begin{align*}
&\max(0, \langle \mu,\alpha\rangle +1-\Phi^+(w\alpha)-\Phi^+(-\alpha)) = \max(0,\ell(x,\alpha)).\qedhere
\end{align*}
\end{proof}
\begin{corollary}[{\cite[Proposition~1.23]{Iwahori1965}}]\label{cor:IwahoriMatsumoto}
Let $x = w\varepsilon^\mu\in \widetilde W$. Then
\begin{align*}
\ell(x) = \sum_{\alpha \in \Phi} \max(0,\ell(x,\alpha)).
\end{align*}
\end{corollary}
\begin{proof}
Use that
$
\ell(x) = \#\{(\alpha,k) \in \Phi_\af^+\mid x\alpha \in \Phi_\af^-\},
$
and decompose the latter set depending on the $\alpha\in \Phi$.
\end{proof}
\begin{corollary}\label{cor:positiveLengthFormula}
Let $x = w\varepsilon^\mu\in \widetilde W$ and $v\in W$. Then
\begin{align*}
\ell(x)\geq \langle v^{-1}\mu,2\rho\rangle - \ell(v) + \ell(wv).
\end{align*}
Equality holds if and only if $v$ is length positive for $x$.
\end{corollary}
\begin{proof}
We calculate
\begin{align*}
\ell(x)\geq&\sum_{\alpha\in \Phi^+}\ell(x,v\alpha)
\\=&\sum_{\alpha\in \Phi^+}\left(\langle \mu,v\alpha\rangle - \Phi^+(-v\alpha) + \Phi^+(-wv\alpha)\right)
\\=&\langle v^{-1}\mu,2\rho\rangle - \ell(v) + \ell(wv).\qedhere
\end{align*}
\end{proof}
\begin{lemma}\label{lem:lengthFunctionalForProducts}
Let $x = w\varepsilon^\mu, x' = w'\varepsilon^{\mu'}\in \widetilde W$ and $\alpha \in \Phi$.
\begin{enumerate}[(a)]
\item $\ell(xx',\alpha) = \ell(x,w'\alpha) + \ell(x',\alpha).$
\item $\ell(x^{-1},\alpha) = -\ell(x,w^{-1}\alpha)$ and $\LP(x^{-1}) = w \LP(x) w_0$.
\end{enumerate}
\end{lemma}
\begin{proof}
\begin{enumerate}[(a)]
\item Note that $xx' = ww'\varepsilon^{(w')^{-1}\mu + \mu'}$ such that
\begin{align*}
&\ell(x,w'\alpha) + \ell(x',\alpha)  \\=& \,\langle \mu,w'\alpha\rangle + \langle \mu',\alpha\rangle - \Phi^+(ww'\alpha) + \Phi^+(w'\alpha) - \Phi^+(w'\alpha) + \Phi^+(\alpha)
\\=&\,\langle (w')^{-1}\mu + \mu',\alpha\rangle -\Phi^+(ww'\alpha) + \Phi^+(\alpha)=\ell(xx',\alpha).
\end{align*}
\item By (a), we have
\begin{align*}
0 = \ell(1,\alpha) = \ell(x x^{-1},\alpha)= \ell(x,w^{-1}\alpha) + \ell(x^{-1},\alpha).
\end{align*}
Now observe that for $v\in W$,
\begin{align*}
v\in \LP(x^{-1})\iff&\forall \beta\in \Phi^+:~\ell(x^{-1},v\beta)\geq 0
\\\iff&\forall \beta\in \Phi^+:~\ell(x^{-1},v(-w_0\beta))\geq 0
\\\iff&\forall\beta\in \Phi^+:~\ell(x,w^{-1} v w_0\beta)\geq 0 \iff v\in w\LP(x) w_0.
\qedhere
\end{align*}
\end{enumerate}
\end{proof}
\begin{lemma}\label{lem:lengthAdditivity}
Let $x = w\varepsilon^\mu, x' = w'\varepsilon^{\mu'}\in \widetilde W$. The following are equivalent:
\begin{enumerate}[(i)]%TFAE
\item $\ell(xx') = \ell(x) + \ell(x')$.
\item For each root $\alpha\in \Phi$, the values $\ell(x,w'\alpha)$ and $\ell(x',\alpha)\in \mathbb Z$ never have opposite signs, i.e.\
\begin{align*}
\ell(x,w'\alpha) \cdot \ell(x',\alpha)\geq 0.
\end{align*}
\item $\left((w')^{-1} \LP(x)\right)\cap \LP(x')\neq \emptyset$.
\end{enumerate}
In this case, $\LP(xx') = \left((w')^{-1} \LP(x)\right)\cap \LP(x')$.
\end{lemma}
\begin{proof}
(i) $\iff$ (ii): By Corollary~\ref{cor:IwahoriMatsumoto} and the equation $\ell(x,\alpha) = -\ell(x,-\alpha)$, we get
\begin{align*}
\ell(xx') =& \sum_{\alpha\in \Phi^+}\abs{\ell(xx',\alpha)}
\\\underset{\text{L\ref{lem:lengthFunctionalForProducts}(a)}}=&\sum_{\alpha\in \Phi^+}\abs{\ell(x,w'\alpha) + \ell(x',\beta)}\\\underset{(\ast)}\leq&\sum_{\alpha\in \Phi^+}\abs{\ell(x,w'\alpha)}+\abs{\ell(x',\alpha)}
\\=&\ell(x) + \ell(x').
\end{align*}
Equality holds at $(\ast)$ iff the values $\ell(x,w'\alpha)$ and $\ell(x',\alpha)$ never have opposite signs. We see that (i) $\iff$ (ii).

(iii) $\Rightarrow$ (ii): Pick $v\in \left((w')^{-1} \LP(x)\right)\cap \LP(x')$. If $\alpha\in \Phi^+$, then both $\ell(x,w'v\alpha)$ and $\ell(x',v\alpha)$ must be non-negative by length positivity. If conversely $\alpha\in \Phi^-$, then both $\ell(x,w'v\alpha)$ and $\ell(x',v\alpha)$ must be non-positive. We see that (ii) holds true.

Finally, let us assume that (ii) holds. It suffices to show that
\begin{align*}
\LP(xx') = \left((w')^{-1} \LP(x)\right)\cap \LP(x'),
\end{align*}
as (iii) follows from this identity. Now for $v\in W$, we have
\begin{align*}
v\in \LP(xx')\iff& \forall \alpha\in \Phi^+:~\ell(xx',v\alpha)\geq 0\\
\underset{\text{L\ref{lem:lengthFunctionalForProducts}(a)}}\iff&\forall \alpha\in \Phi^+:~\ell(x,w'v\alpha) + \ell(x',v\alpha)\geq 0\\
\underset{(ii)}\iff&\forall \alpha\in \Phi^+:~\ell(x,w'v\alpha)\geq 0\text{ and }\ell(x',v\alpha)\geq 0\\
\iff&v\in \left((w')^{-1} \LP(x)\right)\cap \LP(x').\qedhere
\end{align*}
\end{proof}
Given one element $v\in \LP(x)$, one can use it to iteratively enumerate all of $\LP(x)$.
\begin{lemma}\label{lem:LPEnumeration}
Let $x =w\varepsilon^\mu\in \widetilde W$ and $v\in \LP(x)$.
\begin{enumerate}[(a)]
\item For every simple root $\alpha\in \Delta$, we have
\begin{align*}
\ell(x,v\alpha)=0\iff vs_\alpha\in \LP(x).
\end{align*}
\item If the root $\alpha\in \Phi^+$ satisfies $\ell(x,v\alpha)=0$, then there also exists a simple root with this property.
\item Consider the undirected graph $G_{\LP(x)}$ whose vertices are given by $\LP(x)$ and whose edges are of the form $(v,vs_\alpha)$ for $\alpha\in \Delta$ and $v,vs_\alpha\in \LP(x)$. Then $G_{\LP(x)}$ is connected.
\end{enumerate}
\end{lemma}
\begin{proof}
\begin{enumerate}[(a)]
\item If $vs_\alpha\in \LP(x)$, then $\ell(x,v\alpha)$ and $\ell(x,vs_\alpha \alpha) = -\ell(x,v\alpha)$ must both be non-negative. This is only possible if $\ell(x,v\alpha)=0$.

If $\ell(x,v\alpha)=0$, confirm that $\ell(x,v\beta)\geq 0$ for all $\beta\in \Phi^+\cup\{-\alpha\}$. The latter set is preserved by $s_\alpha$.
\item Suppose $\alpha \in \Phi^+\setminus\Delta$ satisfies $\ell(x,v\alpha)=0$. We can write $\alpha=\beta+\gamma$ for positive roots $\beta,\gamma\in \Phi^+$. By length positivity, $\ell(x,v\beta),\ell(x,v\gamma)\geq 0$. If both of these values are $\geq 1$, we get $\ell(x,v\alpha)\geq 1$ by the root functional property. Hence $\ell(x,v\beta)=0$ or $\ell(x,v\gamma)=0$. We can iterate this argument.
\item Let $C\subseteq \LP(x)$ denote the connected component that contains $v$. Among all $v'\in C$, pick one such that $\ell(wv')$ is minimal.

We claim that \begin{align*}\forall\alpha\in \Delta:~\langle \mu,v'\alpha\rangle+\Phi^+(v'\alpha)\geq 1.\tag{$\ast$}\end{align*}
\begin{itemize}
\item If $\ell(x,v'\alpha)=0$, then $v's_\alpha\in C$. Minimality of $\ell(wv')$ ensures $\ell(wv's_\alpha)\geq \ell(wv')$, i.e.\ $wv'\alpha\in \Phi^+$. By definition, $\ell(xv'\alpha)=0$ implies $\langle \mu,v'\alpha\rangle + \Phi^+(v'\alpha)=1$.
\item If $\ell(x,v'\alpha)\geq 1$, we get
\begin{align*}
\langle \mu,v'\alpha\rangle+\Phi^+(v'\alpha)\geq \ell(x,v'\alpha)\geq 1.
\end{align*}
\end{itemize}
Let us re-read condition $(\ast)$: Not only is $(v')^{-1}\mu$ dominant, we have $v'\alpha\in \Phi^+$ for all $\alpha\in \Delta$ with $\langle (v')^{-1}\mu,\alpha\rangle=0$. This describes exactly the length positive element constructed in Example~\ref{ex:usualLPelement}.

To summarize: No matter which connected component of $G_{\LP(x)}$ we consider, it will always contain the one length positive element from Example~\ref{ex:usualLPelement}. Hence $G_{\LP(x)}$ is connected.\qedhere
\end{enumerate}
\end{proof}
We obtain the following description of the shrunken Weyl chambers:
\begin{proposition}\label{prop:shrunkenWeylChambers}
For $x \in \widetilde W$, the following are equivalent:
\begin{enumerate}[(a)]
\item $x$ lies in the lowest two-sided Kazhdan-Lusztig cell of $\widetilde W$.
\item For all $\alpha\in \Phi$, $\ell(x,\alpha)\neq 0$.
\item The set $\LP(x)$ contains only one element.
\end{enumerate}
In this case, we say that $x$ lies in a \emph{shrunken Weyl chamber}.
\end{proposition}
\begin{proof}
The equivalence (1) $\iff$ (2) is well known, cf.\ \cite[Section~3.1]{He2021c}.

The equivalence (2) $\iff$ (3) follows directly from Lemma~\ref{lem:LPEnumeration}.
\end{proof}
\begin{remark}
The length functional presented here is related to the $k$-function from \cite{Shi1987a}. For $w\in W, \mu\in X^\ast(T)$ and $\alpha\in \Phi$, Shi proves
\begin{align*}
k(wt^\mu,\alpha) = \langle \mu,\alpha^\vee\rangle + \Phi^+((\alpha)(w^{-1}))-\Phi^+(\alpha).
\end{align*}
This result is a translation of \cite[Lemma~3.1]{Shi1987a} and \cite[Theorem~3.3]{Shi1987a} into our \enquote{$\Phi^+(\cdot)$}-notation. Up to a few changes of conventions, this recovers exactly our length functional. We will make these changes to express a few of Shi's ideas in terms of the length functional.

Shi classifies the functions $\Phi\rightarrow\mathbb Z$ that are of the form $\ell(x,\cdot)$ in \cite[Proposition~5.1]{Shi1987a}.

Associated to each element $x\in \widetilde W$ and root $\alpha\in \Phi$, he defines $X(x,\alpha)\in \{+,\bigcirc,-\}$ as
\begin{align*}
X(x,\alpha) = \begin{cases}+,&\ell(x,\alpha)>0,\\
\bigcirc,&\ell(x,\alpha)=0,\\
-,&\ell(x,\alpha)<0.
\end{cases}
\end{align*}
The \emph{sign type} of $x$ is defined as $\zeta(x) = (X(x,\alpha))_{\alpha\in \Phi}$. The \emph{admissible} sign types, i.e.\ the image of $\zeta:\widetilde W\rightarrow\{+,\bigcirc,-\}^\Phi$, is explicitly described in \cite[Theorem~2.1]{Shi1987b}. Shi also computes the number of sign types and canonical representatives in $W_a$ for each.

For root systems of type $A_n$, the preimages $\zeta^{-1}(S)$ for the different admissible sign types $S$ form exactly the set of left Kazhdan-Lusztig cells for $W_a$ \cite{Shi1986}. An explicitly described equivalence relation of sign types then classifies the two-sided Kazhdan-Lusztig cells.

The question to fully describe the Kazhdan-Lusztig cells for all affine Weyl groups seems to be open.

The sign type $\zeta(x)$ determines the set of length positive elements for $x$. The converse is not true, i.e.\ it is possible to find groups $G$ and elements $x,y\in \widetilde W$ with $\LP(x) = \LP(y)$ but $\zeta(x)\neq \zeta(y)$. Computer searches have revealed such counterexamples for root systems of types $G_2$ and $B_2$, thus for every non simply-laced root system. For simply-laced root systems, we can prove that the set $\LP(x)$ determines the sign type $\zeta(x)$.
\end{remark}
\begin{proposition}\label{prop:signTypeFromLP}
Assume that $\Phi$ is simply laced, $x \in \widetilde W$ and $\alpha\in \Phi$. Then the following are equivalent:
\begin{enumerate}[(i)]
\item $\ell(x,\alpha)>0$.
\item For all $v\in \LP(x)$, we have $v^{-1}\alpha\in \Phi^+$.
\end{enumerate}
\end{proposition}
\begin{proof}
The implication (i) $\Rightarrow$ (ii) follows from the definition of length positivity.

Now assume (ii). The condition $v^{-1}\alpha\in \Phi^+$ for one $v\in \LP(x)$ already implies $\ell(x,\alpha)\geq 0$. Aiming for a contradiction, we thus assume that $\ell(x,\alpha)=0$.

Recall from Example~\ref{ex:usualLPelement} that there exists an element $v\in \LP(x)$ such that
\begin{align*}
\forall \beta\in \Phi^+:~\langle \mu,v\beta\rangle + \Phi^+(v\beta)\geq 1.
\end{align*}
Considering the case $\beta = v^{-1}\alpha\in \Phi^+$ (by (ii)), we see
\begin{align*}
\ell(x,\alpha) = \langle \mu,v\beta\rangle + \Phi^+(v\beta) - \Phi^+(w\alpha)\geq 1-\Phi^+(w\alpha).
\end{align*}
So if $w\alpha \in \Phi^-$, we conclude (i).

Considering the same situation for $x^{-1}$ by Lemma~\ref{lem:lengthFunctionalForProducts}, we find an element $v\in \LP(x)$ such that
\begin{align*}
\forall \beta\in \Phi^+:~\langle \mu,v\beta\rangle - \Phi^+(wv\beta)\geq 0.
\end{align*}
Considering the case $\beta = v^{-1}\alpha\in \Phi^+$, we see
\begin{align*}
\ell(x,\alpha) = \langle \mu,v\beta\rangle + \Phi^+(\alpha) - \Phi^+(wv\beta)\geq \Phi^+(\alpha).
\end{align*}
So if $\alpha\in \Phi^+$, we are done again.

Let us thus assume from now on that $\alpha\in \Phi^-$ and $w\alpha\in \Phi^+$. In light of the assumption $\ell(x,\alpha)=0$, we can restate this as $\langle \mu,\alpha\rangle = -1$.

For roots $\beta,\gamma\in \Phi$, we write $\beta\leq \gamma$ if the difference $\gamma-\beta$ is a sum of positive roots, and we write $\beta<\gamma$ is moreover $\beta\neq\gamma$.

We define a \emph{root sequence} associated to an element $v\in \LP(x)$ to be a sequence
\begin{align*}
v^{-1}\alpha=\beta_1>\cdots>\beta_\ell\in \Phi^+
\end{align*}
such that $\beta_{i+1}-\beta_i\in\Phi^+$ for $i=1,\dotsc,\ell-1$ and $\langle \mu,v\beta_i\rangle=-1$ for $i=1,\dotsc,\ell$.

Certainly, we can find a root sequence for each $v\in \LP(x)$ of length $1$ by setting $\beta_1 =v^{-1}\alpha$.

We order the set of root sequences lexicographically. Explicitly, consider root sequences $(\beta_1,\dotsc,\beta_\ell)$ associated with $v\in \LP(x)$ and $(\beta_1',\dotsc,\beta'_{\ell'})$ associated with $v'\in \LP(x)$. We write $(\beta_1,\dotsc,\beta_\ell)<(\beta_1',\dotsc,\beta'_{\ell'})$ if one of the following conditions is satisfied:
\begin{itemize}
\item There is $i\in \{1,\dotsc,\min\{\ell,\ell'\}\}$ with $\beta_{i'} = \beta_{i'}'$ for $i'=1,\dotsc,i-1$ and $\beta_i<\beta_i'$.
\item We have $\ell>\ell'$ and $\beta_i = \beta_i'$ for $i=1,\dotsc,\ell'$.
\end{itemize}

Among all possible $v\in \LP(x)$ and root sequences $(\beta_1,\dotsc,\beta_\ell)$ associated with them, we choose a pair such that the root sequence becomes minimal with respect to the above order.

We first claim that $\beta_\ell$ is simple: Indeed, if we had $\beta_\ell = \gamma_1+\gamma_2$ for positive roots $\gamma_1,\gamma_2$, then $\ell(x,v\gamma_1),\ell(v,\gamma_2)\geq 0$ by length positivity. Thus
\begin{align*}
\langle \mu,v\gamma_1\rangle\geq -1,\quad \langle \mu,v\gamma_2\rangle\geq -1,\quad \langle \mu,v\gamma_1+v\gamma_2\rangle = -1.
\end{align*}
Hence $\langle \mu,v\gamma_i\rangle=-1$ for one of the roots $\gamma_1,\gamma_2$. We see that we can extend the root sequence $(\beta_1,\dotsc,\beta_\ell)$, which contradicts minimality by definition.

Note that $\langle \mu,v\beta_\ell\rangle=-1$ and $\ell(x,v\beta_\ell)\geq 0$ implies $\ell(x,v\beta_\ell)=0$. By Lemma~\ref{lem:LPEnumeration}, this means $v' = vs_{\beta_\ell}\in \LP(x)$.

If $\ell=1$, then $(v')^{-1}\alpha = -v^{-1}\alpha$, so we get the desired contradiction to (ii). Therefore, $\ell>1$.

We claim that $\langle \beta_\ell^\vee,\beta_i\rangle\geq 0$ for $i=1,\dotsc,\ell$: Indeed, if we had $\langle \beta_\ell^\vee,\beta_i\rangle<0$, then $\beta_i+\beta_\ell\in \Phi^+$. So we get
\begin{align*}
\ell(x,v(\beta_i+\beta_\ell))\geq 0\text{ and }\langle \mu,v\beta_i + v\beta_\ell\rangle=-2.
\end{align*}
This is impossible.

Note that $\langle \beta_\ell^\vee,\beta_{\ell-1}\rangle=1$, as $\beta_{\ell-1}$ is the sum of $\beta_\ell$ with another root, and $\Phi$ is simply laced.

We thus may pick $\ell'\in\{1,\dotsc,\ell-1\}$ minimally such that $\langle \beta_\ell^\vee,\beta_{\ell'}\rangle>0$. Consider the root sequence
\begin{align*}
\beta_i' = s_{\beta_\ell}(\beta_i),\quad i=1,\dotsc,\ell'.
\end{align*}
This is a root sequence associated with $v' = vs_{\beta_\ell}\in \LP(x)$. Since $\beta_i' = \beta_i$ for $i=1,\dotsc,\ell'-1$ (by choice of $\ell'$), and $\beta_{\ell'}' < \beta_{\ell'}$, it is a smaller root sequence.

This is finally a contradiction to minimality.
\end{proof}
\begin{remark}
The above proof encodes an algorithm, which finds for each root $\alpha\in \Phi$ with $\ell(x,\alpha)=0$ and each $v\in \LP(x)$ a sequence for elements in $\LP(x)$ as in Lemma~\ref{lem:LPEnumeration}. The sequence starts at $v$ and ending in an element $v'\in \LP(x)$ satisfying $(v')^{-1}\alpha\in \Phi^-$. As noted before, this proposition is false for every non simply laced root system.

There is a different way to formulate the above proposition. Deciding automata for reduced words in Coxeter groups have been studied with great interest in the past, cf.\ \cite[Sections 4.8 and 4.9]{Bjorner2005} and the plethora of references provided there. For an affine Weyl group, the nodes of the \emph{canonical automaton} are precisely given by Shi's sign types, cf.\ \cite[Proposition~76]{Ericksson1994}. The nodes of the \emph{minimal automaton} are precisely given by the sets $\LP(x)$ for $x\in\widetilde W$, this follows from Lemma~\ref{lem:lengthAdditivity}. So Proposition~\ref{prop:signTypeFromLP} can be re-phrased as follows: The canonical automaton of an affine Weyl group is minimal if and only if the underlying root system is simply laced.
\end{remark}
\subsection{Quantum Bruhat graph}\label{sec:quantumBruhatGraph}
Associated to the root system $\Phi$, we have the \emph{quantum Bruhat graph} $\QB(W)$ as introduced by Brenti-Fomin-Postnikov \cite{Brenti1998}. This graph and its associated weight function play a crucial role in Section~\ref{chap:generic-sigma-conjugation} when discussing generic $\sigma$-conjugacy classes. In the past, the quantum Bruhat graph was used as a technical tool in a number of different contexts \cite{Postnikov2005, Lam2010, Lenart2015, Milicevic2021, Milicevic2020, Sadhukhan2021, He2021c, He2021d}.
\begin{definition}
\begin{enumerate}[(a)]
\item The \emph{quantum Bruhat graph} associated with $\Phi$, denoted $\QB(W)$, is a $\mathbb Z\Phi^\vee$-weighted directed graph with vertex set $W$. Its edges are of the form $w\rightarrow ws_\alpha$ for $w\in W$ and $\alpha\in \Phi^+$ such that one of the following conditions is satisfied:
\begin{itemize}
\item[(B)] $\ell(ws_\alpha) = \ell(w)+1$ or
\item[(Q)] $\ell(ws_\alpha) = \ell(w) +1-\langle \alpha^\vee,2\rho\rangle$.
\end{itemize}
\item Edges of type (B) are called \emph{Bruhat edges} and have weight $0\in \mathbb Z\Phi^\vee$. Edges of type (Q) are called \emph{quantum edges} and have weight $\alpha^\vee\in \mathbb Z\Phi^\vee$.
\item If $w, w'\in W$, a \emph{path from $w$ to $w'$} is a sequence of adjacent edges in $\QB(W)$
\begin{align*}
p:w = w_1\rightarrow w_2\rightarrow\cdots\rightarrow w_{\ell(p)+1} = w'.
\end{align*}
The \emph{length} of $p$ is the number of edges, denoted $\ell(p)$. The \emph{weight} of $p$ is the sum of its edges weights', denoted $\wt(p)\in \mathbb Z\Phi^\vee$.
\item A path $p$ from $w$ to $w'$ is \emph{shortest} if there is no path $p'$ from $w$ to $w'$ with $\ell(p')<\ell(p)$. In that case, we define $d(w\Rightarrow w') := \ell(p)$.
\end{enumerate}
\end{definition}
\begin{lemma}[{\cite[Lemma~1]{Postnikov2005}}]Let $w, w'\in W$
\begin{enumerate}[(a)]
\item There exists a path from $w$ to $w'$ in $\QB(W)$.
\item Any two shortest paths from $w$ to $w'$ have the same weight, denoted $\wt(w\Rightarrow w')\in \mathbb Z\Phi^\vee$.
\item Any path $p$ from $w$ to $w'$ has weight $\wt(p)\geq \wt(w\Rightarrow w')$.\rightqed
\end{enumerate}
\end{lemma}
One interpretation of the weight function $\wt(w\Rightarrow w')$ is that it measures the failure of the inequality $w\leq w'$ in the Bruhat order on $W$. Indeed, $\wt(w\Rightarrow w')=0$ if and only if $w\leq w'$.

We have the following converse to part (c) of the above lemma.
\begin{lemma}[{\cite[Equation~(4.3)]{Milicevic2020}}]\label{lem:weight2rho}For any path $p$ from $w$ to $w'$, we have
\begin{align*}
\langle \wt(p),2\rho\rangle = \ell(w') - \ell(w) + \ell(p).
\end{align*}
In particular,
\begin{align*}
&\langle \wt(w\Rightarrow w'),2\rho\rangle = \ell(w') - \ell(w) + d(w\Rightarrow w').\rightqed
\end{align*}
\end{lemma}

% !TeX spellcheck = en_GB
\section{$\sigma$-conjugacy classes}\label{chap:sigma-conjugation}
In this section, we review various descriptions of the set $B(G)$ of $\sigma$-conjugacy classes in $G(L)$. This serves mostly as a preparation for the next section, which discusses the \emph{generic} $\sigma$-conjugacy class of an element $x\in \widetilde W$. Throughout this section, we assume that $G$ is quasi-split.

We begin with the classical result of Kottwitz \cite{Kottwitz1985, Kottwitz1997} that shows that the $\sigma$-conjugacy class of an element $g\in G(L)$ is uniquely determined by two invariants. These are called \emph{Kottwitz point} $\kappa(g)\in \pi_1(G)_{\Gamma}=(X_\ast(T)/\mathbb Z\Phi^\vee)_{\Gamma}$ and \emph{(dominant) Newton point} $\nu(g) \in X_\ast(T)_{\Gamma_0}\otimes \mathbb Q$.

If $g$ lies in the normalizer of the maximal torus, $g\in N_G(T)(L)$, then it corresponds to an element in $w\varepsilon^\mu\in \widetilde W$. In this case, $\kappa(g)$ is the image of $\mu$ in $\pi_1(G)_{\Gamma}$.

Viewing both $w$ and $\sigma$ as automorphisms of $X_\ast(T)_{\Gamma_0}$, we write $\sigma\circ w$ for their composition. Let $N\geq 1$ such that the $(\sigma\circ w)^N$ is the identity map. Then $\nu(g)\in X_\ast(T)_{\Gamma_0}\otimes\mathbb Q$ is the unique dominant element in the $W$-orbit of
\begin{align*}
\frac 1N\sum_{k=1}^N (\sigma\circ w)^k\mu.
\end{align*}
It is true, e.g.\ by \cite[Section~3.3]{He2014}, that each $\sigma$-conjugcacy class $[b]\in B(G)$ contains an element of $N_G(T)(L)$, so that the above descriptions of $\kappa(g)$ and $\nu(g)$ actually cover all $\sigma$-conjugacy classes.

In this section, we review a few important results related to these invariants. Our main concern is to bridge the gap between the unramified case, which is often studied in the relevant literature, and the quasi-split case, which we need for our final generalization.

\subsection{Parabolic averages and convex hull}\label{sec:parabolic-averages}
We start by formally defining some averaging functions and proving their basic properties. Neither our results nor our proofs in this section should be too surprising for the educated reader, especially if one keeps the example of $\GL_n$ and its Newton polygons in mind.

Let $N\geq 1$ be an integer such that the action of $\sigma^N$ on $X_\ast(T)$ becomes trivial. Then we define the \emph{$\sigma$-average} of an element $\mu \in X_\ast(T)_{\Gamma_0}\otimes \mathbb Q$ by
\begin{align*}
\avg_\sigma(\mu) := \frac 1N\sum_{k=1}^N \sigma^k(\mu)\in (X_\ast(T)_{\Gamma_0}\otimes \mathbb Q)^{\langle \sigma\rangle}.
\end{align*}
Since $\avg_{\sigma}$ vanishes on terms of the form $\mu -\sigma(\mu)$, it follows that we get a well-defined map $\avg_{\sigma} : X_\ast(T)_\Gamma \rightarrow (X_\ast(T)_{\Gamma_0}\otimes \mathbb Q)^{\langle \sigma\rangle}$.

A similar notion of average is the following: For $J\subseteq \Delta$, denote by $W_J$ the Coxeter subgroup of $W$ generated by the reflections $\{s_\alpha\mid \alpha\in J\}$. For $\mu \in X_\ast(T)_{\Gamma_0}\otimes \mathbb Q$, we define
\begin{align*}
\avg_J(\mu) := \frac 1{\# W_J} \sum_{w\in W_J} w(\mu)\in X_\ast(T)_{\Gamma_0}\otimes \mathbb Q.
\end{align*}
Finally, if $J = \sigma(J)$, we define the function $\pi_J$ by
\begin{align*}
\pi_J := \avg_J \circ \avg_\sigma = \avg_\sigma \circ \avg_J : X_\ast(T)_{\Gamma_0}\otimes \mathbb Q\rightarrow (X_\ast(T)_{\Gamma_0}\otimes \mathbb Q)^{\langle \sigma\rangle}.
\end{align*}
This map was introduced by Chai \cite[Definition~3.2]{Chai2000}. Again, we get an induced map $\pi_J : X_\ast(T)_\Gamma\rightarrow (X_\ast(T)_{\Gamma_0}\otimes \mathbb Q)^{\langle \sigma\rangle}$. If $G$ is split, it can be identified with the \emph{slope map} as introduced by Schieder \cite[Section~2.1.3]{Schieder2015}.

We start with a collection of easy facts on these averages.
\begin{lemma}\label{lem:avgSimpleFacts}
Let $\beta \in X_\ast(T)_{\Gamma}$ and $\mu \in X_\ast(T)_{\Gamma_0}\otimes {\mathbb Q}$. Let $J\subseteq \Delta$ be any subset.
\begin{enumerate}[(a)]
\item For any preimage $\beta'\in X_\ast(T)_{\Gamma_0}$ of $\beta$, we have
\begin{align*}
\langle \beta',2\rho\rangle = \langle \avg_\sigma(\beta),2\rho\rangle.
\end{align*}
In particular, it makes sense to write $\langle \beta,2\rho\rangle$.
\item If $\langle \mu,\alpha\rangle=0$ for all $\alpha\in J$, then $\avg_J(\mu) = \mu$.
\item For all $\alpha\in J$, we have $\langle \avg_J(\mu),\alpha\rangle=0$.
\item If $\mu\geq 0$, then $\avg_J(\mu)\geq 0$.
\item If $\langle \mu,\alpha\rangle \leq 0$ for all $\alpha\in J$, then $\mu\leq w\mu$ for all $w\in W_J$. In particular, $\mu\leq \avg_J(\mu)$.
\end{enumerate}
\end{lemma}
\begin{proof}
(a) follows since $\sigma(2\rho)= 2\rho$ and $\avg_\sigma(b) = \avg_\sigma(b')$.

For (b) and (c), note that the following are equivalent:
\begin{itemize}
\item $\langle \mu,\alpha\rangle=0$ for all $\alpha\in J$,
\item $w(\mu)=\mu$ for all $w\in W_J$.
\end{itemize}
Then both statements follow easily.

For (d), it suffices to only consider the case where $\mu$ is a simple coroot $\mu = \alpha^\vee$. If $\alpha\in J$, then $\avg_J(\mu)=0$. Otherwise $w(\alpha)\in \Phi^+$ for all $w\in W_J$, such that $\avg_J(\mu)>0$.

We prove (e) via induction on $\ell(w)$, the inductive start being clear. If now $\ell(w)\geq 1$ and $w\alpha\in \Phi^-$ for some $\alpha\in J$, then
\begin{align*}
w\mu = (ws_\alpha)(\mu-\langle\mu,\alpha\rangle \alpha^\vee) = (ws_\alpha)\mu + \langle\mu,\alpha\rangle w\alpha^\vee\geq (ws_\alpha)\mu\underset{\text{ind.}}\geq \mu.
\end{align*}
This finishes the induction and the proof.
\end{proof}
\begin{definition}
Let $\mu \in X_\ast(T)_{\Gamma_0}\otimes {\mathbb Q}$ and $J\subseteq \Delta$ be any subset.
\begin{enumerate}[(a)]
\item We say that $J$ is \emph{$\mu$-improving} if we can write $J = \{\alpha_1,\dotsc,\alpha_k\}$ such that
\begin{align*}
\langle \avg_{\{\alpha_1,\dotsc,\alpha_{i-1}\}}(\mu),\alpha_i\rangle \leq 0
\end{align*}
for $i=1,\dotsc,k$.
\item We say that $J$ is \emph{maximally $\mu$-improving} if it is $\mu$-improving, and any $\mu$-improving superset $J'\supseteq J$ satisfies $\avg_J(\mu) = \avg_{J'}(\mu)$.
\end{enumerate}
\end{definition}
E.g.\ any $\mu$-improving subset of maximal cardinality will be maximally $\mu$-improving. Since the empty set is $\mu$-improving, it follows that maximally $\mu$-improving subsets always exist. We make the following immediate observations:
\begin{lemma}\label{lem:convPreparation}
Let $\mu \in X_\ast(T)_{\Gamma_0}\otimes {\mathbb Q}$ and $J\subseteq \Delta$.
\begin{enumerate}[(a)]
\item If $J$ is $\mu$-improving, then $\mu\leq \avg_J(\mu)$.
\item If $J$ is maximally $\mu$-improving, then $\avg_J(\mu)$ is dominant.
\item If $c\in X_\ast(T)_{\Gamma_0}\otimes \mathbb Q$ is dominant and $\mu\leq c$, then \begin{align*}&\avg_J(\mu)\leq \avg_J(c)\leq c.\rightqed\end{align*}
\end{enumerate}
\end{lemma}
If follows that there is a uniquely determined maximum
\begin{align*}
\conv'(\mu) := \max_{J\subseteq \Delta} \avg_J(\mu),
\end{align*}
and that $\conv'(\mu) = \avg_J(\mu)$ for every maximally $\mu$-improving $J$. We define
\begin{align*}
\conv(\mu) := \conv' (\avg_\sigma(\mu)),\qquad \mu\in X_\ast(T)_{\Gamma_0}\otimes\mathbb Q\text{ or }\mu\in X_\ast(T)_{\Gamma}.
\end{align*}
\begin{example}
For the split group $G = \GL_n$, the operations $\conv$ and $\conv'$ agree. Drawing elements of $X_\ast(T)\otimes {\mathbb Q}$ as polygons, the function $\conv$ corresponds to taking the upper convex hull (hence its name).
\end{example}
\begin{lemma}\label{lem:convFacts}
Let $\mu \in X_\ast(T)_{\Gamma_0}\otimes {\mathbb Q}$.
\begin{enumerate}[(a)]
\item The value $\conv'(\mu)$ is the uniquely determined element $c\in X_\ast(T)_{\Gamma_0}$ satisfying the following three conditions:\begin{itemize}
\item $\mu\leq c$,\item $c$ is dominant and\item $c = \avg_J(\mu)$ for some $J\subseteq \Delta$.
\end{itemize}
\item If $\mu'\in X_\ast(T)_{\Gamma_0}\otimes \mathbb Q$ satisfies $\mu\leq \mu'$, then $\conv'(\mu)\leq \conv'(\mu')$.
\item Write
\begin{align*}
\conv'(\mu) - \mu =& \sum_{\alpha\in \Delta} c_\alpha \alpha^\vee,
\\J_1:=&\{\alpha\in \Delta\mid c_\alpha\neq 0\},
\\J_2 :=&\{\alpha\in \Delta\mid\langle \conv'(\mu),\alpha\rangle=0\}.
\end{align*}
For any subset $J\subseteq \Delta$, we have
\begin{align*}
\conv'(\mu) = \avg_J(\mu)\iff J_1\subseteq J\subseteq J_2.
\end{align*}
\item There exists $J\subseteq \Delta$ with $\sigma(J) = J$ and $\conv(\mu) = \pi_J(\mu)$. In particular,
\begin{align*}
\conv(\mu) = \max_{\substack{J\subseteq \Delta\\\sigma(J) = J}} \pi_J(\mu).
\end{align*}
\item Let $J\subseteq \Delta$ such that
\begin{align*}
\forall \alpha \in \Phi^+\setminus \Phi_J^+:~\langle \mu,\alpha\rangle\geq 0.
\end{align*}
Then there exists $J'\subseteq J$ with $\conv'(\mu) = \avg_{J'}(\mu)$. In other words, the set $J_1$ from (c) is a subset of $J$.
\end{enumerate}
\end{lemma}
\begin{proof}
(a) and (b) are immediate.
\begin{enumerate}[(a)]
\addtocounter{enumi}{2}
\item Let us first consider a subset $J\subseteq \Delta$ with $\conv'(\mu) = \avg_J(\mu)$. Then $\conv'(\mu) - \mu \in \mathbb Q\Phi_J^\vee$ by definition of $\avg_J(\mu)$. We see that $J_1\subseteq J$ must hold. Similarly, $\langle \conv'(\mu),\alpha\rangle=0$ for all $\alpha\in J$ by Lemma~\ref{lem:avgSimpleFacts}. Thus we must have $J_1\subseteq J\subseteq J_2$.

We show that $\avg_{J_1}(\mu)$ is dominant. Let $\alpha\in \Delta$. If $\alpha\in J_1$, then $\langle \avg_{J_1}(\mu),\alpha\rangle=0$ by Lemma~\ref{lem:avgSimpleFacts}. So let us assume that $\alpha\in \Delta\setminus J_1$. Because $\avg_{J_1}(\mu)\leq \conv'(\mu)$ and $\avg_{J_1}(\mu)\equiv \mu\equiv \conv'(\mu)\pmod{\mathbb Q\Phi_{J_1}^\vee}$, we can write
\begin{align*}
\conv'(\mu) - \avg_{J_1}(\mu) = \sum_{\beta\in J_1} c_\beta' \beta^\vee,\quad c_\beta'\in \mathbb Q_{\geq 0}.
\end{align*}
Now we get
\begin{align*}
\langle \avg_{J_1}(\mu),\alpha\rangle=\underbrace{\langle \conv'(\mu),\alpha\rangle}_{\geq 0} +\sum_{\beta\in J_1} \underbrace{c_\beta'\langle -\beta^\vee,\alpha\rangle}_{\geq 0}\geq 0.
\end{align*}
This shows that $\avg_{J_1}(\mu)$ is dominant.

If $J$ is chosen such that $\conv'(\mu) = \avg_J(\mu)$, then
\begin{align*}
\conv'(\mu)\geq \avg_{J_1}(\mu)\underset{\text{L\ref{lem:avgSimpleFacts}}}\geq \avg_J\avg_{J_1}(\mu)\underset{J_1\subseteq J}=\avg_J(\mu) = \conv'(\mu).
\end{align*}
Thus $\avg_{J_1}(\mu) = \conv'(\mu)$.

So if for any intermediate set $J_1\subseteq J\subseteq J_2$, we obtain
\begin{align*}
\avg_J(\mu) = \avg_J(\avg_{J_1}(\mu)) = \avg_J(\conv'(\mu))\underset{J\subseteq J_2}=\conv'(\mu).
\end{align*}
\item Replacing $\mu$ by $\avg_\sigma(\mu)$, we may certainly assume $\mu \in (X_\ast(T)_{\Gamma_0}\otimes \mathbb Q)^{\sigma}$. Since $\mu = \sigma(\mu)$, we conclude $\conv'(\mu) = \sigma(\conv'(\mu))$. Then we can choose $J$ be either of the sets $J_1$ or $J_2$ from (c).

Now the \enquote{in particular} part is easy to see.
\item Let $J'\subseteq J$ be a $\mu$-improving subset such that there is no $\mu$-improving subset $J'\subsetneq J''\subseteq J$. By Lemma~\ref{lem:convPreparation}, $\mu\leq \avg_{J'}(\mu)$. It suffices to show that $\avg_{J'}(\mu)$ is dominant. Seeing $\mu$ as a coweight for the root system $\Phi_J$, the set $J'$ is maximally $\mu$-improving from this perspective, so $\langle \avg_{J'}\mu,\alpha\rangle\geq 0$ for all $\alpha\in \Phi_J^+$.

If $\alpha\in \Phi^+\setminus \Phi_J^+$, then $w\alpha\in \Phi^+\setminus \Phi_J^+$ for all $w\in W_J$, such that
\begin{align*}
\langle \avg_{J'}(\mu),\alpha\rangle = \frac 1{\# W_{J'}} \sum_{w\in W_{J'}}\underbrace{\langle \mu,w\alpha\rangle}_{\geq 0}\geq 0.
\end{align*}
Here, we use the assumption made on $\mu$ and $J$. 

As $\avg_{J'}(\mu)$ is dominant, we get the desired result by (a).
\qedhere
\end{enumerate}
\end{proof}
As an immediate application, let us describe Newton points of elements in $\widetilde W$ with this language:
\begin{definition}
For $w\in W$,we  write $\supp(w)\subseteq \Delta$ for the set of all simple roots whose corresponding simple reflections occur in some/every reduced expression for $w$. Define $\supp_\sigma(w) := \bigcup_{n\in \mathbb Z}\sigma^n(\supp(w))$.
\end{definition}
\begin{lemma}\label{lem:newtonPointsAsAverages}
Let $x = w\varepsilon^\mu\in \widetilde W$ and $N>0$ such that $(\sigma\circ w)^N=\mathrm{id}$. Pick $v\in W$ such that
\begin{align*}
v^{-1}\frac 1N\sum_{k=1}^N (\sigma\circ w)^k(\mu)\in X_\ast(T)_{\Gamma_0}\otimes \mathbb Q
\end{align*}
becomes dominant. Let $J = \supp_{\sigma}(v^{-1}\presig(wv))$. Then
\begin{align*}
\nu(x) = \pi_J(v^{-1}\mu).
\end{align*}
\end{lemma}
\begin{proof}
Straightforward calculation. For an alternative proof, cf.\ \cite[Proposition~4.1]{Chai2000}.
\end{proof}
\subsection{$\lambda$-invariant and defect}\label{sec:defect}
For this section, we fix a $\sigma$-conjugacy class $[b]\in B(G)$. Following Hamacher-Viehmann \cite[Lemma/Definition~2.1]{Hamacher2018}, we define its \emph{$\lambda$-invariant} by
\begin{align*}
\lambda_G(b) := \max\{\tilde\lambda\in X_\ast(T)_{\Gamma}\mid \avg_\sigma(\tilde\lambda)\leq\nu(b)\text{ and }\kappa(b) = \lambda +\mathbb Z\Phi^\vee\text{ in }\pi_1(G)_{\Gamma}\}.
\end{align*}
While the article of Hamacher-Viehmann assumes the group to be unramified, the construction of $\lambda_G(b)$ works without changes for quasi-split $G$.

Let us write
\begin{align*}
\nu(b) - \avg_\sigma(\lambda_G(b)) =& \sum_{\alpha\in \Delta}c_\alpha\alpha^\vee,\\
J_1:=&\{\alpha\in \Delta\mid c_\alpha\neq 0\},\\
J_2:=&\{\alpha\in \Delta\mid\langle \nu(b),\alpha\rangle=0\}.
\end{align*}

We have the following simple observations:
\begin{lemma}\label{lem:lambdaSimpleFacts}
\begin{enumerate}[(a)]
\item Pick $\mu \in X_\ast(T)_{\Gamma}$ and $J\subseteq \Delta$ with $J = \sigma(J)$ such that $\nu(b) = \pi_J(\mu)$ and $\kappa(b) = \mu + \mathbb Z\Phi^\vee \in \pi_1(G)_{\Gamma}$.
Then \begin{align*}\nu(b) = \pi_J(\lambda_G(b)) = \conv(\lambda_G(b)).\end{align*}
\item We have $J_1\subseteq J_2$.
For $J\subseteq\Delta$ with $\sigma(J) = J$,
\begin{align*}
\nu(b) = \pi_{J}(\lambda_G(b))\iff J_1\subseteq J\subseteq J_2.
\end{align*}
\end{enumerate}
\end{lemma}
\begin{proof}
\begin{enumerate}[(a)]
\item Choose a lift $\tilde \mu\in X_\ast(T)_{\Gamma_0}$. Then
\begin{align*}
\pi_J(\mu) = \pi_J(\tilde \mu) = \avg_\sigma\sum_{w\in W_J} w\tilde \mu.
\end{align*}
We can choose an element $w\in W_J$ such that $w\tilde\mu$ becomes anti-dominant with respect to the roots in $J$, i.e.\ $\langle w\tilde \mu,\alpha\rangle\leq 0$ for all $\alpha\in J$. Then $\pi_J(\tilde \mu)=\pi_J(w\tilde \mu)\geq w\tilde \mu$ by Lemma~\ref{lem:avgSimpleFacts}.

In particular, the image of $w\tilde \mu$ in $X_\ast(T)_{\Gamma}$ is $\leq\lambda_G(b)$ by construction of $\lambda_G(b)$.  Thus
\begin{align*}
\nu(b) = \pi_{J}(w\tilde \mu)\leq \pi_{J}(\lambda_G(b))\leq \conv(\lambda_G(b)).
\end{align*}
Since $\avg_\sigma(\lambda_G(b))\leq \nu(b)$ and $\nu(b)$ is dominant, we use Lemma~\ref{lem:convPreparation} to see that $\conv(\lambda_G(b))\leq \nu(b)$. Hence $\nu(b) = \conv(\lambda_G(b)) = \pi_{J}(\lambda_G(b))$.
\item By \cite[Section~3.3]{He2014}, $b = [x]$ for some $x\in \widetilde W$. Applying Lemma~\ref{lem:newtonPointsAsAverages} to $x$, we see that $\mu$ and $J$ exist as in (a). In particular, $\nu(b) = \conv(\lambda_G(b))$.

Now all claims follow from Lemma~\ref{lem:convFacts}.\qedhere
\end{enumerate}
\end{proof}

Related to the notion of the $\lambda$-invariant is the notion of \emph{defect} of an element $[b]\in B(G)$.

Following \cite[Proposition~6.2]{Kottwitz1985}, we fix an element $x = w\varepsilon^\mu$ of length zero in the extended affine Weyl group $\widetilde W_{J_2}$ of the Levi subgroup of $G$ associated with $J_2$ such that $[b] = [x]\in B(G)$.

We denote by $J_b$ the $\sigma$-twisted centralizer of $b\in G(L)$, i.e.\ the reductive group over $F$ with $F$-valued points
\begin{align*}
J_b(F) = \{g\in G(L)\mid g^{-1} b\sigma(g) = b\}.
\end{align*}
Then the defect of $[b]$ has the following equivalent descriptions:
\begin{proposition}\label{prop:defect}
The following non-negative integers all agree. The common value is called the \emph{defect} of $[b]$, denoted $\defect(b)$.
\begin{enumerate}[(i)]
\item $\dim (X_\ast(T)_{\Gamma_0}\otimes\mathbb Q)^{\sigma} -\dim (X_\ast(T)_{\Gamma_0}\otimes\mathbb Q)^{\sigma w}$,
\item $\rk_F(G) - \rk_F(J_b)$,
\item $\langle \nu(b),2\rho\rangle-\langle\lambda_G(b),2\rho\rangle$,
\item $\#(J_1/\sigma)$, the number of $\sigma$-orbits in $J_1$,
\item $\min_{v\in W}\ell(v^{-1}\presig(w v))$,
\item $\min_{v\in W_{J_1}}\ell(v^{-1}\presig (wv))$.
\end{enumerate}
\end{proposition}
The notion of defect was originally defined in \cite[Equation~1.9.1]{Kottwitz2006} for split groups, using the expression in (i). Kottwitz shows the equality with (ii) as \cite[Theorem~1.10.1]{Kottwitz2006} and the equality with (iii) as \cite[Theorem~1.9.2]{Kottwitz2006}.

If $G$ is not split, the expression of (ii) is commonly used as definition. In the unramified case, the equality of (ii) with (iii) is then known as \cite[Proposition~3.8]{Hamacher2015}, and Hamacher's proof shows the equality with (i) and (iv).

For the remainder of this section, we sketch how to prove Proposition~\ref{prop:defect} for quasi-split groups $G$. The main idea is a reduction to the superbasic case.
\begin{lemma}
Assume that $[b]$ is superbasic.
Denote by $n = \#(\Delta/\sigma)$ the number of $\sigma$-orbits in $\Delta$.
\begin{enumerate}[(a)]
\item We have
\begin{align*}
(X_\ast(T)_{\Gamma_0}\otimes\mathbb Q)^{\sigma w} = \{\mu\in X_\ast(T)_{\Gamma_0}\otimes\mathbb Q\mid \sigma(\mu)=\mu\text{ and }\langle \mu,\Phi\rangle = \{0\}\}.
\end{align*}
In particular,
\begin{align*}
n = \dim (X_\ast(T)_{\Gamma_0}\otimes\mathbb Q)^{\sigma}-(X_\ast(T)_{\Gamma_0}\otimes\mathbb Q)^{\sigma w}.
\end{align*}
\item We have
\begin{align*}n = \min_{v\in W}\ell(v^{-1}\presig(wv)).
\end{align*}
More precisely, we find $v\in W$ and a subset $\Delta'\subseteq \Delta$ such that $\# \Delta' = n$ and $v^{-1}\presig(wv)$ is a Coxeter element for $\Delta'$.
\item We have
\begin{align*}
n=
\langle \nu(b) - \avg_\sigma(\lambda_G(b)),2\rho\rangle.
\end{align*}
\end{enumerate}
\end{lemma}
\begin{proof}
Superbasic elements only exist if each irreducible component of $\Phi$ is a root system of type $A$.

All claims may certainly be checked individually on each $\sigma$-connected component, so to lighten our notation, we will assume that $\Delta$ is $\sigma$-connected.
\begin{enumerate}[(a)]
\item If $\mu\in X_\ast(T)_{\Gamma_0}\otimes\mathbb Q$ is $\sigma$-stable and orthogonal to all roots, it is certainly fixed by $\sigma w$. Let conversely $\mu\in X_\ast(T)_{\Gamma_0}\otimes\mathbb Q$ satisfy $\sigma w(\mu)=\mu$. Then we find $v\in W$ such that $v\mu\in X_\ast(T)_{\Gamma_0}\otimes\mathbb Q$ is dominant. Observe that
\begin{align*}
\left(v\presig( w v^{-1})\right) \sigma v\mu = v\sigma w v^{-1} v\mu = v\mu.
\end{align*}
Since $\sigma v\mu$ is dominant and in the $W$-orbit of $v\mu$, we get $\sigma v\mu=v\mu$. In particular, the dominant coweight $v\mu$ gets stabilized by $v\presig(wv^{-1})\in W$.

Let $J:=\Stab(v\mu)$ denote the stabilizer of the dominant coweight $v\mu$. Then $J=\sigma(J)$, so $J$ defines a $\sigma$-stable Levi subgroup of $G$. Its extended affine Weyl group $\widetilde W_J$ contains $v^{-1}\presig(xv)$, so $b$ comes from a $\sigma$-conjugacy class in this Levi subgroup. This is only possible if $J=\Delta$, i.e.\ $\langle v\mu,\Phi\rangle=\{0\}$. In particular, $v\mu = v^{-1}(v\mu)=\mu$, proving the claim.
\item
Decompose the Dynkin diagram of $\Delta$ into connected components, written as $\Delta = C_1\sqcup \dotsc\sqcup C_k$, such that $\sigma(C_i) = C_{i+1}$ for $i=1,\dotsc,k-1$ and $\sigma(C_k) = C_1$. Let $W_C := W_{C_1}$ denote the Weyl group of $C:=C_1$.

Note that each $C_i$ is of type $A_n$ with $n$ as given. Write $C_\af$ for the affine Dynkin diagram associated with $C = C_1$. Then the action of $\sigma^k$ on $C_\af$ must fix the special node, and be either the identity or the unique involution on the complement, i.e.\ $C$. The element $x\presig x\cdots\prescript{\sigma^{k-1}}{}x$, being an element of length zero in the affine Weyl group of $C$, acts on $C_\af$ by some cyclic permutation. The composition of these two maps, $(\sigma\circ x)^k$, should act transitively on $C_\af$.

One quickly checks that this is only possible if $\sigma^k$ is the identity map on $C_\af$.

Now write $w = w_1 \presig(w_2)\cdots\prescript{\sigma^{k-1}}{}(w_k)$ with $w_1,\dotsc,w_k\in W_C$. Let $v_1\in W_C$ and define
\begin{align*}
v := v_1\presig(v_2)\cdots\prescript{\sigma^k}{}(v_k)\in W,\qquad v_{i+1} = w_i v_i\text{ for }i=1,\dotsc,k-1.
\end{align*}
Then
\begin{align*}
v^{-1} \presig(wv) =& v_1^{-1}\presig(v_2^{-1})\cdots\prescript{\sigma^k}{}(v_k^{-1})\cdot (w_kv_k)\presig(w_1v_1)\cdots \cdots\prescript{\sigma^{k-1}}{}(w_{k-1}v_{k-1})
\\=&v_1^{-1} w_k v_k = v_1^{-1} w_k\cdots w_1 v_1\in W_C.
\end{align*}
We know that $W_C$ is a Coxeter group of type $A_n$, so a symmetric group. It is a classical result that each element in a symmetric group is conjugate to a Coxeter element for a parabolic subgroup. In other words, we find $v_1$ and $\Delta'\subseteq C$ such that $v_1^{-1}w_k\cdots w_1 v_1$ is a Coxeter element of $\Delta'$.

In particular, we get
\begin{align*}
n=\#C\geq\#\Delta' = \ell(v^{-1}\presig(wv))\geq \#\supp(v^{-1}\presig(wv))\underset{\text{superbasic}}\geq n.
\end{align*}
Thus $\#\Delta'=n$.
\item
It remains to evaluate 
\begin{align*}
\langle \nu(b) - \avg_\sigma(\lambda_G(b)),2\rho\rangle = \sum_{\alpha\in \Delta}2c_\alpha.
\end{align*}
This calculation is carried out by Hamacher \cite[Section~3]{Hamacher2015}, and we obtain the value $n$ as claimed. The equality only depends on the affine root system together with the $\sigma$-action, so the fact that Hamacher only considers unramified groups is irrelevant. While his argument using characters of finite group representations is very elegant, one can also obtain the same result in a more straightforward manner with explicit calculations of Newton polygons (as we are in the $A_n$ case).\qedhere\end{enumerate}
\end{proof}
\begin{proof}[Proof of Proposition~\ref{prop:defect}]
The equality of (i) with (ii) is a standard Bruhat-Tits theoretic argument, cf.\ \cite[Section~4.3]{Kottwitz2006} or \cite[Proof of Prop.~3.8]{Hamacher2015}.

Observe that the values of (i), (iii), (iv) and (vi) do not change if we pass to the Levi subgroup of $G$ defined by $J_1$. If we do so, $[b]$ becomes a superbasic $\sigma$-conjugacy class. Then the equalities of (i), (iii), (iv) and (vi) follow immediately from the preceding lemma.

It remains to show that, in the general case, (v) agrees with (vi). Suppose this was not the case. Then we would find some $v\in W$ such that
\begin{align*}
\ell(v^{-1} \presig(wv))<\#(J_1/\sigma).
\end{align*}
Consider the element $y = v^{-1}\presig(xv)\in \widetilde W$ and the subset $J\subseteq \Delta$ given by $J:=\supp_\sigma(v^{-1} \presig(wv))$. Then $J$ defines a $\sigma$-stable Levi subgroup $M\subseteq G$ such that $[b]$ has a preimage in $B(M)$. This is only possible if $J_1\subseteq J$, so $J = J_1$. But we must have
\begin{align*}
\ell(v^{-1} \presig(wv))\geq \#\supp(v^{-1}\presig(wv)) \geq \#(J/\sigma) = \#(J_1/\sigma),
\end{align*}
contradiction!
\end{proof}
\subsection{Fundamental elements}\label{sec:fundamental-elements}
Recall the equivalent characterizations of fundamental elements:
\begin{proposition}\label{prop:fundamental}For $x = w\varepsilon^\mu\in\widetilde W$, the following are equivalent:
\begin{enumerate}[(i)]
\item $\ell(x) = \langle\nu(x),2\rho\rangle$.
\item For all $n\geq 1$, we have
\begin{align*}
\ell(x\cdot \presig x\cdots \prescript{\sigma^{n-1}}{}x) = n\ell(x).
\end{align*}
\item There exist $v\in \LP(x)$ and a $\sigma$-stable $J\subseteq \Delta$ such that $v^{-1}\presig(wv)\in W_J$ and for all $\alpha\in \Phi_J$, we have $\ell(x,v\alpha)=0$.
\item For every orbit $O\subseteq \Phi$ with respect to the action of $(\sigma\circ w)$ on $\Phi$, we have
\begin{align*}
\left(\forall \alpha \in O:~\ell(x,\alpha)\geq 0\right)\text{ or }\left(\forall \alpha\in O:~\ell(x,\alpha)\leq 0\right).
\end{align*}
\end{enumerate}
If $G$ is defined over $\mathcal O_F$, this is moreover equivalent to
\begin{enumerate}[(i)]
\addtocounter{enumi}{4}
\item Every element $y\in IxI$ is of the form $y = i^{-1} x\presig i$ for some $i\in I$.
\end{enumerate}
If these equivalent conditions are satisfied, we call $x$ \emph{fundamental}.\rightqed
\end{proposition}
Let us first discuss the unramified case. In this case, the equivalence of (i) and (ii) is due to He \cite[Lemma~8.1]{He2010}. Elements satisfying these conditions are called \emph{good} in \cite{He2010} and \emph{$\sigma$-straight} in more recent literature. Condition (iii) is a reformulation of the notion of \emph{fundamental $(J,w,\delta)$-alcoves} from Goertz-He-Nie \cite[Section~3.3]{Goertz2015}. Condition (v) is the notion of fundamental elements from \cite{Goertz2010}. The equivalence of (i), (iii) and (v) is a result of Nie \cite{Nie2015}. Condition (iv) is new, but we will not need it in the sequel.

If $G$ is quasi-split but not unramified, the cited proofs fail because the map $X_\ast(T)_{\Gamma_0}\rightarrow X_\ast(T)_{\Gamma_0}\otimes\mathbb Q$ might no longer be injective. It is conceivable that the proofs might be generalized with a bit of work. Instead, we sketch how to prove the equivalences of (i)--(iv) using our language of length functionals, where issues with the torsion part of $X_\ast(T)_{\Gamma_0}$ are non-existent.
\begin{proof}[Proof of Proposition~\ref{prop:fundamental}]
Lemma~\ref{lem:lengthAdditivity} implies the equivalence of (ii) and (iv). Moreover, the implication (iii) $\implies$ (iv) is immediate.

Let $N>0$ such that the action of $(\sigma\circ w)^N$ on $X_\ast(T)_{\Gamma_0}$ becomes trivial. For any $v\in W$ and $\alpha\in \Phi$, we calculate
\begin{align*}
&\left\langle \frac 1Nv^{-1}\sum_{k=1}^N(\sigma\circ w)^k\mu,\alpha\right\rangle
\\=&\frac 1N\sum_{k=1}^N \langle \mu,(\sigma\circ w)^k v\alpha\rangle
\\=&\frac 1N\sum_{k=1}^N \langle \mu,(\sigma\circ w)^k v\alpha\rangle + \Phi^+((\sigma\circ w)^k v\alpha) - \Phi^+((\sigma\circ w)^{k+1} v\alpha)
\\=&\frac 1N\sum_{k=1}^N \ell(x,(\sigma\circ w)^kv\alpha).
\end{align*}
Pick now $v\in W$ such that $v^{-1}\sum_{k=1}^N (\sigma\circ w)^k\mu = \nu(x)$. Then
\begin{align*}
\langle \nu(x),2\rho\rangle = \sum_{\alpha\in \Phi^+}\frac 1N\sum_{k=1}^N \ell(x,(\sigma\circ w)^kv\alpha)\geq \ell(x).
\end{align*}
Equality holds if and only if $(\sigma\circ w)^kv\in \LP(x)$ for all $k\in \mathbb Z$. If we define $J := \supp_\sigma(v^{-1}\presig(wv))$, we see that (ii) implies (iii).

It remains to show that (iv) implies (ii). This follows directly from the above calculation.
\end{proof}

Fundamental elements play an important role for our description of generic $\sigma$-conjugacy classes. If $x$ is fundamental, the generic $\sigma$-conjugacy class $[b_x]$ coincides with the $\sigma$-conjugacy class of $x$, whose Newton and Kottwitz points are easily computed. The $\lambda$-invariant and the defect of $[x]$ however are less straightforward to see. For now, we compute the defect. This will later help to compute the $\lambda$-invariant, in view of
\begin{align*}
\defect([x]) = \langle \nu(x),2\rho\rangle - \langle \lambda(x),2\rho\rangle\qquad \text{(Proposition~\ref{prop:defect})}.
\end{align*}
\begin{lemma}\label{lem:fundamentalDefect}
Let $x$ be fundamental, and choose $v\in \LP(x)$ and $J\subseteq \Delta$ as in Proposition~\ref{prop:fundamental} (iii).
\begin{enumerate}[(a)]
\item Every $v'\in vW_J$ is length positive for $x$. Moreover, $(x, v', J)$ also satisfies condition (iii) of Proposition~\ref{prop:fundamental}.
\item If $v\in W^J$, then $(\prescript{\sigma^{-1}}{}v)^{-1}xv$ coincides with an element of length zero in the extended affine Weyl group $\widetilde W_J = W_J\ltimes X_\ast(T)_{\Gamma_0}$.
\item The defect of $x$ is given by
\begin{align*}
\defect([x]) = \min_{v'\in vW_J} \ell( (v')^{-1}\presig(wv')) = \min_{v'\in W}\ell( (v')^{-1}\presig(wv)).
\end{align*}
\end{enumerate}
\end{lemma}
\begin{proof}
\begin{enumerate}[(a)]
\item This is a very straightforward calculation.
\item By definition, $(\prescript{\sigma^{-1}}{}v)^{-1}xv\in\widetilde W_J$. The length calculation is straightforward using Lemma~\ref{lem:lengthFunctionalForProducts}. For an alternative proof concept, cf.\ \cite[Proposition~3.2]{He2014b}.
\item In view of (a), we may assume $v\in W^J$. Then
\begin{align*}
\defect([x])=\defect\left([(\prescript{\sigma^{-1}}{}v)^{-1}xv]\right)
\end{align*}
By (b), the element $(\prescript{\sigma^{-1}}{}v)^{-1}xv\in \widetilde W$ satisfies the conditions needed to compute its defect using Proposition \ref{prop:defect} (v) and (vi). The claim follows.\qedhere
\end{enumerate}
\end{proof}
In order to reduce claims about arbitrary elements in $\widetilde W$ to fundamental ones, we need the following lemma. If $G$ is unramified, this is a classical result of Viehmmann \cite[Proposition~5.5]{Viehmann2014}.
\begin{lemma}\label{lem:nonEmptynessBruhatCondition}
Let $x\in \widetilde W$ and $[b]\in B(G)_x$, i.e.\ $[b]\in B(G)$ with $X_x(b)\neq\emptyset$. Then there exists a fundamental element $y\in \widetilde W$ such that $y\leq x$ in the Bruhat order and $[y] = [b]$ in $B(G)$.
\end{lemma}
\begin{proof}
Induction by $\ell(x)$. We distinguish a number of cases.
\begin{enumerate}[1.]
\item Suppose that $x$ is of minimal length in its $\sigma$-conjugacy class in $\widetilde W$ and that $x = uy$ for some fundamental $y\in \widetilde W$ with $\ell(x) = \ell(u) + \ell(y)$ and $[x] = [y]$.

By \cite[Theorem~3.5]{He2014}, $[b] = [x]$ so that $y\leq x$ satisfies the desired conditions.
\item Suppose that there exists a simple affine reflection $s\in S_\af$ such that $\ell(sx\presig s)<\ell(x)$. By the \enquote{Deligne-Lusztig reduction method} of Goertz-He \cite[Corollary~2.5.3]{Goertz2010b}, we must have $[b] \in B(G)_{x'}$ for $x'=sx\presig s$ or $x'=sx$. By induction, we get an element $y\leq x'$ with the desired properties. Since $x'< x$, the claim follows.
\item In general, we find by \cite[Theorem~3.4]{He2014b} a sequence of elements
\begin{align*}
x = x_1,\dotsc,x_n\in \widetilde W
\end{align*}
such that
\begin{itemize}
\item $x_{i+1} = s_i x_i \presig s_i$ for some simple reflection $s_i\in S$ ($i=1,\dotsc, n-1$),
\item $\ell(x_{i}) = \ell(x)$ for $i=1,\dotsc,n$ and
\item $x_n$ satisfies condition 1.\ or 2.
\end{itemize}
In particular, we find $y'\leq x_n$ fundamental with $[y'] = [b]$.

By \cite[Lemma~2.3]{Nie2015}, there exists $y\leq x$ with $\ell(y)\leq \ell(y')$ and $y$ being $\sigma$-conjugate to $y'$ in $\widetilde W$. While Nie's proof only covers unramified groups, this statement is purely about combinatorics of root systems and affine Weyl groups, so the generalization to quasi-split groups is immediate.

Now observe that $[y] = [y'] = [b]\in B(G)$. In particular,
\begin{align*}
\langle \nu(b),2\rho\rangle\leq \ell(y) \leq \ell(y') = \langle \nu(b),2\rho\rangle.
\end{align*}
We see that $y$ must be fundamental as well.
\end{enumerate}
In any case, the claim follows, finishing the induction and the proof.
\end{proof}
% !TeX spellcheck = en_GB
\section{Generic $\sigma$-conjugacy class}\label{chap:generic-sigma-conjugation}
For an element $x\in \widetilde W$, the \emph{generic} $\sigma$-conjugacy class $[b] = [b_x]\in B(G)$ is the uniquely determined $\sigma$-conjugacy class such that $IxI\cap [b]$ is dense in $IxI$. For each $y\in \widetilde W$, we write $[y]\in B(G)$ for the $\sigma$-conjugacy class of any representative of $y$ in $G(L)$. We have the following description due to Viehmann:
\begin{theorem}[{\cite[Corollary~5.6]{Viehmann2014}}]\label{thm:truncations}
Let $x \in \widetilde W$. Then $[b_x]$ is the largest $\sigma$-conjugacy class in $B(G)$ of the form $[y]$ where $y\leq x$ in the Bruhat order on $\widetilde W$.\rightqed
\end{theorem}
Viehmann's original proof makes the assumption that the group under consideration is unramified, but it is not hard to remove this assumption. Indeed, we saw in Lemma~\ref{lem:nonEmptynessBruhatCondition} that \cite[Proposition~5.5]{Viehmann2014} can be proved without this assumption, and then Viehmann's proof of \cite[Corollary~5.6]{Viehmann2014} works without further changes.

We can now describe this generic $\sigma$-conjugacy class more explicitly:
\begin{theorem}\label{thm:genericGKP}Assume that $G$ is quasi-split.
Let $x = w\varepsilon^\mu\in \widetilde W$ and denote by $[b_x]$ its generic $\sigma$-conjugacy class. Writing $\lambda_x := \lambda_G(b_x)$, we have
\begin{align*}
\lambda_x = \max_{v\in W} \left(v^{-1}\mu - \wt(v\Rightarrow\presig(wv))\right)\in X_\ast(T)_{\Gamma}.
\end{align*}
\end{theorem}
We call $\lambda_x$ the \emph{generic $\lambda$-invariant} of $x$. We discuss previous works and some applications of this result now, before giving its proof in the next subsection.

If $G$ is split and $\mu$ sufficiently regular, Mili\'cevi\'c \cite{Milicevic2021} proves that
\begin{align*}
\nu_x = \lambda_x = v^{-1}\mu-\wt(v\Rightarrow wv)
\end{align*}
for the unique element $v\in \LP(x)$.
This was the first paper to derive an explicit formula for $\nu_x$ from Theorem~\ref{thm:truncations}. Her result has since been generalized by Sadhukhan \cite{Sadhukhan2021}, improving the regularity assumption significantly.
He and Nie \cite[Proposition~3.1]{He2021c} proved the same formula whenever $x$ is in a shrunken Weyl chamber, and even a generalization for non split groups $G$.
\ifthesis\else

These previous results allow for a short proof of the following essential statement on the quantum Bruhat graph.
\begin{lemma}\label{lem:weightEstimate}
Let $w\in W$ and $\alpha\in \Phi^+$. Then
\begin{align*}
\wt(ws_\alpha\Rightarrow w)\leq \alpha^\vee\Phi^+(w\alpha).
\end{align*}
\end{lemma}
\begin{proof}
If $w\alpha\in \Phi^-$, then $ws_\alpha\leq w$ in the Bruhat order of $W$, showing $\wt(ws_\alpha\Rightarrow w)=0$. Let us hence assume that $w\alpha\in \Phi^+$.

We may assume that $G$ is split and pick a dominant and sufficiently regular coweight $\mu\in X_\ast(T)_{\Gamma_0}$. By Mili\'cevi\'c's result, the generic Newton point of $x := s_{w\alpha} \varepsilon^{ws_\alpha(\mu)}$ is given by
\begin{align*}
\nu_x = \mu - \wt(ws_\alpha\Rightarrow w).
\end{align*}
Observe that
\begin{align*}
\varepsilon^{ws_\alpha(\mu-\alpha^\vee)} = s_{w\alpha} \varepsilon^{-w\alpha^\vee}x<x
\end{align*}
in the Bruhat order. By Theorem~\ref{thm:truncations}, we get
$
\nu_x \geq \mu - \alpha^\vee,
$
proving the claim.
\end{proof}
\fi

We use it to provide a more explicit way of calculating generic $\lambda$-invariants. The following lemma does not depend on the theorem, while the corollary does.
\begin{lemma}\label{lem:genericGKPImprovements}
Let $x=w\varepsilon^\mu\in \widetilde W$ and $v\in W$.
\begin{enumerate}[(a)]
\item If $v$ is not length positive for $x$, and $vs_\alpha$ is an adjustment, then
\begin{align*}
v^{-1}\mu - \wt(v\Rightarrow\presig(wv))\leq (vs_\alpha)^{-1}\mu-\wt(vs_\alpha\Rightarrow\presig(wvs_\alpha))
\end{align*}
as elements in $X_\ast(T)_{\Gamma}$.
\item We have
\begin{align*}
\langle v^{-1}\mu - \wt(v\Rightarrow\presig(wv)),2\rho\rangle \leq \ell(x)-d(v\Rightarrow\presig(wv)).
\end{align*}
Equality holds if and only if $v\in \LP(x)$.
\end{enumerate}
\end{lemma}
\begin{proof}
\begin{enumerate}[(a)]
\item We compute in $X_\ast(T)_{\Gamma}$ using Lemma~\ref{lem:weightEstimate}:
\begin{align*}
(vs_\alpha)^{-1}\mu - \wt(vs_\alpha\Rightarrow \presig(wvs_\alpha))\geq&v^{-1}\mu - \langle \mu,v\alpha\rangle\alpha^\vee -\wt(vs_\alpha\Rightarrow v) \\&- \wt(v\Rightarrow \presig(wv)) - \wt(\presig(wv)\Rightarrow \presig(wvs_\alpha))
\\\geq &v^{-1}\mu - \langle \mu,v\alpha\rangle\alpha^\vee -\Phi^+(v\alpha)\alpha^\vee \\&- \wt(v\Rightarrow \presig(wv)) - \Phi^+(-wv\alpha)\alpha^\vee
\\=&v^{-1}\mu- \wt(v\Rightarrow \presig(wv)) - (\ell(x,v\alpha)+1)
\\\geq&v^{-1}\mu- \wt(v\Rightarrow \presig(wv)).
\end{align*}
\item Indeed, using Corollary~\ref{cor:positiveLengthFormula} and Lemma~\ref{lem:weight2rho}, we obtain
\begin{align*}
\langle v^{-1}\mu - \wt(v\Rightarrow\presig(wv)),2\rho\rangle=&
\langle v^{-1}\mu,2\rho\rangle - \ell(v) + \ell(wv) - d(v\Rightarrow\presig(wv))
\\\leq&\ell(x) - d(v\Rightarrow\presig(wv)),
\end{align*}
with equality iff $v\in \LP(x)$.\qedhere
\end{enumerate}
\end{proof}
\begin{corollary}\label{cor:genericGKPMinDistance}
Let $x = w\varepsilon^\mu\in \widetilde W$. Among all elements $v\in \LP(x)$, pick one such that the distance $d(v\Rightarrow\presig(wv))$ in the quantum Bruhat graph becomes minimal. Then
\begin{align*}
\lambda_x = v^{-1}\mu - \wt(v\Rightarrow\presig(wv)) \in X_\ast(T)_{\Gamma}.
\end{align*}
In particular, the generic Newton point of $x$ is given by \begin{align*}\nu_x = \conv(v^{-1}\mu - \wt(v\Rightarrow\presig(wv))).\end{align*}
\end{corollary}
\begin{proof}
We know that $\lambda_x = (v')^{-1}\mu - \wt(v'\Rightarrow \presig(wv'))$ for some $v'\in W$ by the theorem. Using the above lemma, we conclude that the same equality holds for some $v'\in \LP(x)$.

Now $v^{-1}\mu - \wt(v\Rightarrow\presig(wv))\leq (v')^{-1}\mu - \wt(v'\Rightarrow \presig(wv'))$ by the theorem, and
\begin{align*}
\langle v^{-1}\mu - \wt(v\Rightarrow\presig(wv)),2\rho\rangle\geq \langle (v')^{-1}\mu - \wt(v'\Rightarrow \presig(wv')),2\rho\rangle
\end{align*}
by choice of $v$. The claim follows.
\end{proof}
In order to compute the generic Newton point $\nu_x = \max_J \pi_J(\lambda_x)$, we can give the following estimate on the required set $J$.
\begin{lemma}\label{lem:gnpJEstimate}
Let $x = w\varepsilon^\mu\in \widetilde W$, $v\in \LP(x)$ and $J\subseteq \Delta$ such that $J=\sigma(J)$ and
\begin{align*}
\forall \alpha\in \Phi^+\setminus \Phi^+_J:~\ell(x,v\alpha)>0.
\end{align*}
Then there exists $J'\subseteq J$ with $\sigma(J') = J'$ and
\begin{align*}
\conv(v^{-1}\mu - \wt(v\Rightarrow\presig(wv))) = \pi_{J'}(v^{-1}\mu - \wt(v\Rightarrow\presig(wv))).
\end{align*}
\end{lemma}
\begin{proof}
In view of Lemma \ref{lem:convFacts} (e), it suffices to show for each $\alpha\in \Phi^+\setminus \Phi^+_J$ that
\begin{align*}
\langle \avg_\sigma(v^{-1}\mu - \wt(v\Rightarrow\presig(wv))),\alpha\rangle\geq 0.
\end{align*}
Let $N>1$ such that the action of $\sigma^N$ on $X_\ast(T)_{\Gamma_0}$ becomes trivial. Then
\begin{align*}
&\langle \avg_\sigma(v^{-1}\mu - \wt(v\Rightarrow\presig(wv))),\alpha\rangle\\=&\frac 1N\sum_{k=1}^N\langle v^{-1}\mu - \wt(v\Rightarrow\presig(wv)),\sigma^k(\alpha)\rangle
\\=&\frac 1N\sum_{k=1}^N\langle \mu,v\sigma^k(\alpha)\rangle -\langle \wt(v\Rightarrow\presig(wv)),\sigma^k(\alpha)\rangle.
\end{align*}
By\footnote{The original formulation of this statement has a small typo, the version cited here is the correct one: Indeed, let $x,y\in W$ and define dominant coweights $\mu_1, \mu_2\in X_\ast(T)_{\Gamma_0}$ on each simple root $\alpha\in \Delta$ as follows:
\begin{align*}
\langle \mu_1,\alpha\rangle := \Phi^+(-y^{-1}\alpha),\quad \langle\mu_2,\alpha\rangle := \Phi^+(x\alpha).
\end{align*}
Then one checks easily that we are in the situation of \cite[Theorem~1.1]{He2021c}, and part (1) of this theorem yields $\langle\wt(y^{-1}\Rightarrow x),\alpha\rangle\leq \Phi^+(-y^{-1}\alpha) + \Phi^+(x\alpha)$.} \cite[Section~2.5]{He2021c}, we may estimate
\begin{align*}
\langle \wt(v\Rightarrow\presig(wv)),\sigma^k(\alpha)\rangle\leq& \Phi^+(-v\sigma^k(\alpha)) + \Phi^+(\presig(wv)\sigma^k(\alpha)) \\=& \Phi^+(-v\sigma^k(\alpha)) +\Phi^+(wv\sigma^{k-1}(\alpha)).
\end{align*}
Thus
\begin{align*}
&\frac 1N\sum_{k=1}^N\left(\langle \mu,v\sigma^k(\alpha)\rangle -\langle \wt(v\Rightarrow\presig(wv)),\sigma^k(\alpha)\rangle\right)
\\\geq&\frac 1N\sum_{k=1}^N\left(\langle \mu,v\sigma^k(\alpha)\rangle -\Phi^+(-v\sigma^k(\alpha)) -\Phi^+(wv\sigma^{k-1}(\alpha))\right)
\\=&\frac 1N\sum_{k=1}^N\left(\langle \mu,v\sigma^k(\alpha)\rangle -\Phi^+(-v\sigma^k(\alpha)) -\Phi^+(wv\sigma^{k}(\alpha))\right)
\\=&\frac 1N\sum_{k=1}^N\Bigl(\underbrace{\ell(x,v\sigma^k(\alpha))}_{\geq 1}-1\Bigr)\geq 0.
\end{align*}
This finishes the proof.
\end{proof}
In particular, if $x$ lies in a shrunken Weyl chamber with $\LP(x) = \{v\}$, we may set $J=\emptyset$ and obtain
\begin{align*}
\nu_x = \avg_\sigma(v^{-1}\mu-\wt(v\Rightarrow\presig(wv)))\in X_\ast(T)_{\Gamma_0}\otimes\mathbb Q.
\end{align*}
This recovers the aforementioned results of Mili\'cevi\'c, Sadhuhkhan and He-Nie as consequences of our Theorem~\ref{thm:genericGKP}.

Finally, we classify the \emph{cordial} elements from Mili\'cevi\'c-Viehmann \cite{Milicevic2020}.
\begin{definition}
Let $x =w\varepsilon^\mu\in \widetilde W$ and $v\in W$ be the specific length positive element constructed in Example~\ref{ex:usualLPelement}. Then $x$ is cordial if
\begin{align*}
\ell(x)-\ell(v^{-1}\presig(wv))= \langle \nu_x,2\rho\rangle - \defect(b_x).
\end{align*}
\end{definition}
\begin{proposition}\label{prop:cordial}
Let $x = w\varepsilon^\mu\in \widetilde W$ and $v\in \LP(x)$. Then
\begin{align*}
\ell(x) - \ell(v^{-1}\presig(wv)) \leq \langle \nu_x,2\rho\rangle - \defect(b_x).
\end{align*}
Equality holds if and only if both conditions (a) and (b) are satisfied. Moreover, the condition (a) is always equivalent to (a').
\begin{itemize}
\item[(a)] The generic $\lambda$-invariant $\lambda_x$ is given by
\begin{align*}
\lambda_x = v^{-1}\mu - \wt(v\Rightarrow\presig(wv))\in X_\ast(T)_{\Gamma}.
\end{align*}
\item[(a')] We have
\begin{align*}
d(v\Rightarrow\presig(wv)) = \min_{v'\in \LP(x)} d(v'\Rightarrow\presig(wv')).
\end{align*}
\item[(b)] We have $d(v\Rightarrow\presig(wv)) = \ell(v^ {-1}\presig(wv))$.
\end{itemize}
\end{proposition}
\begin{proof}
By Lemma~\ref{lem:genericGKPImprovements} and Theorem~\ref{thm:genericGKP}, (a) $\iff$ (a').

For the remaining claims, we calculate
\begin{align*}
\ell(x) - \ell(v^{-1}\presig(wv)) \leq& \ell(x) - d(v\Rightarrow\presig(wv))
\\=&\langle v^{-1}\mu-\wt(v\Rightarrow\presig(wv)),2\rho\rangle
\\\underset{\text{T\ref{thm:genericGKP}}}\leq&\langle \lambda_x,2\rho\rangle \underset{\text{P\ref{prop:defect}}}=  \langle \nu_x,2\rho\rangle - \defect(b_x).\qedhere
\end{align*}
\end{proof}
\begin{corollary}\label{cor:cordial}
Let $x = w\varepsilon^\mu\in \widetilde W$ and $v\in W$ be of minimal length such that $v^{-1}\mu$ is dominant. Then $x$ is cordial if and only if the following two conditions are both satisfied:
\begin{enumerate}[(1)]
\item For each $v'\in \LP(x)$, $d(v\Rightarrow\presig(wv))\leq d(v'\Rightarrow\presig(wv'))$.
\item $d(v\Rightarrow\presig(wv)) = \ell(v^{-1}\presig(wv))$.\rightqed
\end{enumerate}
\end{corollary}
This corollary generalizes the description of superregular cordial element for split $G$ due to Mili\'cevi\'c-Viehmann \cite[Proposition~4.2]{Milicevic2020} and the description of shrunken cordial elements due to He-Nie \cite[Remark~3.2]{He2021c}. One can generalize the statement and proof of \cite[Theorem~1.2~(b),(c)]{Milicevic2020} accordingly.
\subsection{Proof of the Theorem}
Fix $x=w\varepsilon^\mu\in \widetilde W$. We need to show the following two claims:
\begin{itemize}
\item There exists some $v\in W$ such that
\begin{align*}
\lambda_x \leq v^{-1}\mu - \wt(v\Rightarrow\presig(wv))\in X_\ast(T)_{\Gamma}.
\end{align*}
\item For each $v\in W$, we have
\begin{align*}
v^{-1}\mu - \wt(v\Rightarrow\presig(wv)) \leq \lambda_x\in X_\ast(T)_{\Gamma}.
\end{align*}
By definition of $\lambda_G(x)$, this is equivalent to
\begin{align*}
\avg_\sigma(v^{-1}\mu - \wt(v\Rightarrow\presig(wv)))\leq \nu_x\in X_\ast(T)_{\Gamma_0}\otimes\mathbb Q.
\end{align*}
\end{itemize}
Let us use the shorthand notation $\lambda \leq^\sigma \lambda'$ to say that the image of $\lambda$ in $X_\ast(T)_\Gamma$ is less than or equal to the image of $\lambda'$ in $X_\ast(T)_\Gamma$ ($\lambda, \lambda'$ being elements of $X_\ast(T), X_\ast(T)_{\Gamma_0}$ or $X_\ast(T)_\Gamma$).

We write $\lambda\equiv^\sigma \lambda'$ to denote $\lambda\leq^\sigma\lambda'$ and $\lambda'\leq^\sigma \lambda$. Similarly, we write $\lambda<^\sigma\lambda'$ to denote $\lambda\leq^\sigma\lambda'$ but $\lambda'\not\leq^\sigma \lambda$.

For this section, call an element $v\in W$ \emph{maximal} if there exists no $v'\in W$ such that
\begin{align*}
v^{-1}\mu - \wt(v\Rightarrow\presig(wv))<^\sigma(v')^{-1}\mu - \wt(v'\Rightarrow\presig(wv')).
\end{align*}
\begin{lemma}\label{lem:maximalityConsequences}
Let $v\in W$ be maximal. Moreover, fix a root $\alpha\in \Phi^+$ such that
\begin{align*}
\wt(v\Rightarrow\presig(wv)) \equiv^\sigma \alpha^\vee\Phi^+(-v\alpha) + \wt(vs_\alpha\Rightarrow \presig(wv)).
\end{align*}
Then precisely one of the following conditions is satisfied:
\begin{enumerate}[(1)]
\item We have $\ell(x,v\alpha)>0$ and the element
\begin{align*}
x' := w'\varepsilon^{\mu'} := xr_{v\alpha,\Phi^+(-v\alpha)}\in \widetilde W
\end{align*}
satisfies $x'<x$ and
\begin{align*}
(vs_\alpha)^{-1}\mu' - \wt(vs_\alpha\Rightarrow \presig(w'vs_\alpha)) \equiv^\sigma v^{-1}\mu -\wt(v\Rightarrow \presig(wv)).
\end{align*}
\item We have $\ell(x,v\alpha)=0$, $vs_\alpha\in W$ is maximal with
\begin{align*}
v^{-1}\mu - \wt(v\Rightarrow\presig(wv))\equiv^\sigma 
(vs_\alpha)^{-1}\mu - \wt(vs_\alpha \Rightarrow\presig(wvs_\alpha))
\end{align*}
and
\begin{align*}
\wt(vs_\alpha\Rightarrow \presig(wvs_\alpha)) \equiv^\sigma \wt(vs_\alpha\Rightarrow \presig(wv)) + \alpha^\vee\Phi^+(-wv\alpha).
\end{align*}
\end{enumerate}
\end{lemma}
\begin{remark}
If $v\neq \presig(wv)$ and $v\rightarrow vs_\alpha$ is an edge in $\QB(W)$ that is part of a shortest path from $v$ to $\presig(wv)$, then the root $\alpha\in \Phi^+$ will satisfy the condition of the Lemma.
\end{remark}
\begin{proof}[Proof of Lemma~\ref{lem:maximalityConsequences}]
We use maximality of $v$ by comparing to $vs_\alpha$. Now calculate
\begin{align*}
&(vs_\alpha)^{-1}\mu - \wt(vs_\alpha\Rightarrow\presig(wvs_\alpha))
\\\geq&(vs_\alpha)^{-1}\mu - \wt(vs_\alpha\Rightarrow \presig(wv)) -
 \wt(\presig(wv)\Rightarrow \presig(wvs_\alpha)).
\\\equiv^\sigma &(vs_\alpha)^{-1}\mu + \alpha^\vee \Phi^+(-v\alpha) - \wt(v\Rightarrow\presig(wv)) - \underbrace{\wt(wv\Rightarrow wvs_\alpha)}_{\leq \alpha^\vee \Phi^+(wv\alpha)\text{ by \weightEstimateShort}}
\\\geq&v^{-1}\mu - \langle \mu,v\alpha\rangle \alpha^\vee + \alpha^\vee\Phi^+(-v\alpha) -  \wt(v\Rightarrow\presig(wv)) - \alpha^\vee\Phi^+(wv\alpha)
\\=&v^{-1}\mu -\wt(v\Rightarrow\presig(wv)) -\ell(x,v\alpha)\alpha^\vee.
\end{align*}
If $\ell(x,v\alpha)<0$, we get a contradiction to the maximality of $v$.

Next assume that $\ell(x,v\alpha)=0$. Then every inequality in the above computation must be an equality (up to $\sigma$-coinvariants), or we would again get a contradiction. In particular, $vs_\alpha$ must be maximal, as \begin{align*}
v^{-1}\mu - \wt(v\Rightarrow\presig(wv))\equiv^\sigma 
(vs_\alpha)^{-1}\mu - \wt(vs_\alpha \Rightarrow\presig(wvs_\alpha)).
\end{align*}
Moreover, we obtain
\begin{align*}
\wt(vs_\alpha\Rightarrow\presig(wvs_\alpha)) \equiv^\sigma \wt(vs_\alpha\Rightarrow \presig(wv)) + \alpha^\vee(-wv\alpha).
\end{align*}
This shows all the claims in (2).

Finally assume $\ell(x,v\alpha)>0$. Then $x'<x$, i.e.\ $x(v\alpha, \Phi^+(-v\alpha))\in \Phi^-_{\af}$, follows from Lemma~\ref{lem:lengthFunctionalAsCountingAffineRoots}. Calculating explicitly, we get
\begin{align*}
w' \varepsilon^{\mu'} = w\varepsilon^\mu s_{v\alpha} \varepsilon^{\Phi^+(-v\alpha)v\alpha^\vee}
=ws_{v\alpha}\varepsilon^{s_{v\alpha}(\mu) + \Phi^+(-v\alpha)v\alpha^\vee}.
\end{align*}
So indeed,
\begin{align*}
(vs_\alpha)^{-1}\mu' - \wt(vs_\alpha\Rightarrow \presig(w'vs_\alpha)) =& v^{-1}\mu - \alpha^\vee\Phi^+(-v\alpha) - \wt(vs_\alpha\Rightarrow\presig(wv))
\\\equiv^\sigma&v^{-1}\mu - \wt(v\Rightarrow\presig(wv)).\qedhere
\end{align*}
\end{proof}
\begin{corollary}\label{cor:genericGKPDichotomy}
Let $v$ be maximal. Then at least one of the following conditions is satisfied:
\begin{enumerate}[(1)]
\item There exists $x' = w'\varepsilon^{\mu'}<x$ and $v'\in W$ such that
\begin{align*}
v^{-1}\mu - \wt(v\Rightarrow \presig(wv)) \equiv^\sigma (v')^{-1}\mu' - \wt(v'\Rightarrow \presig(w'v')).
\end{align*}
\item The element $\presig(wv)\in W$ is maximal, and we have
\begin{align*}
v^{-1}\mu - \wt(v\Rightarrow \presig(wv))~\equiv^\sigma~ \presig(wv)^{-1}\mu - \wt(\presig(wv)\Rightarrow\presig(w\presig(wv))).
\end{align*}
\end{enumerate}
\end{corollary}
\begin{proof}
Choose a shortest path in $\QB(W)$
\begin{align*}
p : v\rightarrow vs_{\alpha_1}\rightarrow vs_{\alpha_1}s_{\alpha_2}\rightarrow\cdots\rightarrow vs_{\alpha_1}s_{\alpha_2}\cdots s_{\alpha_k} = \presig(wv).
\end{align*}
Consider the roots
\begin{align*}
\beta_i = vs_{\alpha_1}\cdots s_{\alpha_{i-1}}(\alpha_i)\in \Phi,\qquad i=1,\dotsc,k.
\end{align*}
We fix $i^\ast \in \{0,\dotsc,k\}$ maximally such that $\ell(x,\beta_i)=0$ for $1\leq i\leq i^\ast$.

We claim that each $v_i$ for $i=0,\dotsc,i^\ast$ satisfies the following conditions:
\begin{enumerate}[(a)]
\item $v_i$ is maximal,
\item $d(v_i\Rightarrow\presig(wv_i)) = d(v_i\Rightarrow\presig(wv)) + d(\presig(wv)\Rightarrow\presig(wv_i))$.
\item $v^{-1}\mu -\wt(v\Rightarrow\presig(wv)) \equiv^\sigma v_i^{-1}\mu-\wt(v_i\Rightarrow\presig(wv_i))$.
\end{enumerate}
Induction on $i$. Since $v_0=v$, the claim is clear for $i=0$. Now in the inductive step,
assume that $i<i^\ast$ and that the conditions (a)--(c) are true for $v_i$. We apply Lemma~\ref{lem:maximalityConsequences} to $(v_i,\alpha_i)$. This is possible, as $v_i\rightarrow v_{i+1}$ is part of a shortest path from $v_i$ to $\presig(wv)$ (by choice of the path $p$), hence part of a shortest path from $v_i$ to $\presig(wv_i)$ by (b).

Since $i<i^\ast$, we get $\ell(x,v_i\alpha_i)=0$, so condition (2) of Lemma~\ref{lem:maximalityConsequences} must be satisfied. Now (a) and (c) follow immediately for $v_{i+1}$. For condition (b), use condition (2) of the lemma to compute
\begin{align*}
&\wt(v_{i+1}\Rightarrow\presig(wv_{i+1}))\\\equiv^\sigma &\wt(v_{i+1}\Rightarrow \presig(wv_i))+\alpha_i^\vee\Phi^+(-wv_i\alpha_i)
\\\underset{\text{(b)}}=&\wt(v_{i+1}\Rightarrow \presig(wv)) + \wt(\presig(wv)\Rightarrow \presig(wv_i)) + \alpha_i^\vee\Phi^+(-wv_i\alpha_i)
\\\underset{\text{\weightEstimateShort}}{\geq^\sigma}&\wt(v_{i+1}\Rightarrow \presig(wv)) + \wt(\presig(wv)\Rightarrow \presig(wv_i)) +\wt(\presig(wv_i)\Rightarrow\presig(wv_{i+1})
\\\geq&\wt(v_{i+1}\Rightarrow \presig(wv)) + \wt(\presig(wv)\Rightarrow \presig(wv_{i+1}))
\\\geq&\wt(v_{i+1}\Rightarrow\presig(wv_{i+1})).
\end{align*}
We see that equality must hold in every step (up to the $\sigma$-action). In light of Lemma~\ref{lem:weight2rho}, condition (b) for $v_{i+1}$ follows, finishing the induction.

With the above claim proved for all $i\in\{0,\dotsc,i^\ast\}$, we distinguish two cases:
\begin{enumerate}[(1)]
\item Case $i^\ast<k$. Then $\ell(x, \beta_{i^\ast+1}) = \ell(x, v_{i^\ast}(\alpha_{i^\ast+1}))>0$ by choice of $i^\ast$. Applying Lemma~\ref{lem:maximalityConsequences} to $v_{i^\ast}$ and $\alpha_{i^\ast+1}$, we immediately get the desired $x'$.
\item Case $i^\ast =k$. Then $\presig(wv) = v_{i^\ast}$ and we obtain everything claimed.\qedhere
\end{enumerate}
\end{proof}
\begin{lemma}\label{lem:genericKottwitzLowerBound}
Let $v\in W$. Then there exists some $x'\leq x$ with
\begin{align*}
\nu(x')\geq \avg_{\sigma}(v^{-1}\mu - \wt(v\Rightarrow \presig(wv))).
\end{align*}
In other words, $\lambda_x\geq^\sigma v^{-1}\mu - \wt(v\Rightarrow \presig(wv))$.
\end{lemma}
\begin{proof}
Induction on $\ell(x)$. We may certainly assume that $v$ is maximal. If there exists  an element $x' = w'\varepsilon^{\mu'}<x$ and $v'\in W$ with
\begin{align*}
v^{-1}\mu - \wt(v\Rightarrow\presig(wv))~\equiv^\sigma~ (v')^{-1}\mu'-\wt(v'\Rightarrow\presig(w'v')),
\end{align*}
we may apply the inductive hypothesis to $x'$ and are done.

Let us assume that this is not the case. By the above corollary, we see that $\presig(wv)$ is maximal and 
\begin{align*}
v^{-1}\mu - \wt(v\Rightarrow \presig(wv))~\equiv^\sigma ~\presig(wv)^{-1}\mu - \wt(\presig(wv)\Rightarrow\presig(w\presig(wv))).
\end{align*}
For $n\geq 0$, we define the element $v_n\in W$ by $v_0 := v$ and $v_{n+1} := \presig(wv_n)\in W$. A simple induction argument shows that each $v_n$ is maximal and
\begin{align*}
v^{-1}\mu - \wt(v\Rightarrow \presig(wv))\equiv^\sigma v_n^{-1}\mu - \wt(v_n\Rightarrow \presig(wv_n)).
\end{align*}
We calculate for $\lambda\in X_\ast(T)_{\Gamma_0}$:
\begin{align*}
v_{n}\lambda = \presig(wv_{n-1})\lambda = \sigma\circ wv_{n-1}\left(\sigma^{-1}\lambda\right) = (\sigma\circ w)^n v(\sigma^{-n} \lambda).
\end{align*}
Thus
\begin{align*}
v_n^{-1}\lambda = \sigma^n v^{-1}(\sigma\circ w)^{-n}(\lambda).
\end{align*}
Let $N\geq 1$ such that the action of $(\sigma\circ w)^N$ on $X_\ast(T)$ becomes trivial. We see that
\begin{align*}
\avg_\sigma\left(v^{-1}\mu - \wt(v\Rightarrow\presig(wv))\right) =& \frac 1N\sum_{n=1}^N \avg_\sigma\left(v_n^{-1}\mu - \wt(v_n\Rightarrow\presig(wv_n))\right)
\\\leq&\frac 1N\sum_{n=1}^N \avg_\sigma\left(v_n^{-1}\mu \right)
\\=&\frac 1N\sum_{n=1}^N \avg_\sigma\left(v^{-1}(\sigma\circ w)^{-n}\mu \right)
\\=&\avg_\sigma v^{-1}\frac 1N\sum_{n=1}^N (\sigma\circ w)^{-n}\mu.
\\\leq&\avg_\sigma \nu(x) = \nu(x).
\end{align*}
Thus we may choose $x'=x$, finishing the induction and the proof.
\end{proof}
\begin{lemma}\label{lem:genericKottwitzFundamental}
Let $x=w\varepsilon^\mu\in \widetilde W$ be a fundamental element, and choose $v'\in \LP(x)$ with $\defect([x]_\sigma) = \ell((v')^{-1}\presig(wv'))$ as in Lemma~\ref{lem:fundamentalDefect}. Then
\begin{align*}
\lambda_x \equiv^\sigma (v')^{-1}\mu - \wt(v'\Rightarrow\presig(wv')).
\end{align*}
\end{lemma}
\begin{proof}
By Lemma~\ref{lem:genericKottwitzLowerBound}, we have
\begin{align*}
\lambda_x \geq^\sigma (v')^{-1}\mu - \wt(v'\Rightarrow\presig(wv')).
\end{align*}
Now we calculate
\begin{align*}
&\langle \lambda_x- (v')^{-1}\mu + \wt(v'\Rightarrow\presig(wv')),2\rho\rangle
\\\underset{\text{L\ref{lem:genericGKPImprovements}}}=&\langle \lambda_x,2\rho\rangle - \ell(x) + d(v'\Rightarrow\presig(wv'))
\\\underset{\text{fund.}}=&\langle \lambda_G(x),2\rho\rangle - \langle \nu(x),2\rho\rangle+ d(v'\Rightarrow\presig(wv'))
\\\underset{\text{P\ref{prop:defect}}}=&-\defect([x]_\sigma) + d(v'\Rightarrow\presig(wv'))
\\\leq&-\defect([x]_\sigma) + \ell((v')^{-1}\presig(wv')) \underset{\text{assump.}}=0.
\end{align*}
The inequality on the last line is \cite[Lemma~4.3]{Milicevic2020}.
\end{proof}
\begin{lemma}\label{lem:genericKottwitzUpperBound}
There exists $v\in W$ such that
\begin{align*}
\lambda_x\leq^\sigma v^{-1}\mu - \wt(v\Rightarrow\presig(wv)).
\end{align*}
\end{lemma}
\begin{proof}
Induction on $\ell(x)$.

Let us first consider the case that there exists an element $x' = w'\varepsilon^{\mu'}<x$ with $[b_{x'}] = [b_x]\in B(G)$. If this is the case, we may further assume by definition of the Bruhat order that $x' = xr_a$ for some affine root $a\in \Phi_\af^+$.

Using the induction assumption, we find some $v'\in W$ such that
\begin{align*}
\lambda_{x'} = \lambda_x\leq^\sigma (v')^{-1}\mu'-\wt(v'\Rightarrow\presig(w'v')).
\end{align*}
Write $a = (\alpha,k)$ such that $w' = ws_\alpha$ and $\mu' = s_\alpha(\mu)+k\alpha^\vee$. The condition $\ell(x')<\ell(x)$ means that $xa \in \Phi_\af^-$, which we can rewrite as
\begin{align*}
k-\langle \mu,\alpha\rangle<\Phi^+(w\alpha).
\end{align*}
We distinguish the following cases.
\begin{itemize}
\item Case $(v')^{-1}\alpha \in \Phi^-$. Define $v := s_\alpha v'$ and compute
\begin{align*}
\lambda_x\leq^\sigma~&(v')^{-1}\mu'-\wt(v'\Rightarrow\presig(w'v'))
\\=&v^{-1}(\mu - k\alpha^\vee) - \wt(s_\alpha v \Rightarrow \presig(wv))
\\\leq&v^{-1}\mu - kv^{-1}\alpha^\vee - \wt(v\Rightarrow \presig(wv)) + \wt(v\Rightarrow s_\alpha v)
\\\underset{\text{\weightEstimateShort}}\leq&v^{-1}\mu - kv^{-1}\alpha^\vee - \wt(v\Rightarrow \presig(wv)) + v^{-1}\alpha^\vee \Phi^+(-\alpha)
\\=&v^{-1}\mu - \wt(v\Rightarrow \presig(wv)) + (\Phi^+(-\alpha)-k)v^{-1}\alpha^\vee
\\\leq&v^{-1}\mu - \wt(v\Rightarrow \presig(wv)).
\end{align*}
The inequality on the last line follows since $\Phi^+(-\alpha)-k\leq 0$ (as $a\in \Phi_\af^+$) and $v^{-1}\alpha\in \Phi^+$ by assumption.
\item Case $(v')^{-1}\alpha \in \Phi^+$. Define $v := v'$ and compute
\begin{align*}
\lambda_x\leq^\sigma~&(v')^{-1}\mu'-\wt(v'\Rightarrow\presig(w'v'))
\\=&v^{-1}(\mu - \langle\mu,\alpha\rangle \alpha^\vee+ k\alpha^\vee) - \wt(v \Rightarrow ws_\alpha v)
\\\leq^\sigma &v^{-1}\mu + (-\langle \mu,\alpha\rangle +k)v^{-1}\alpha^\vee - \wt(v\Rightarrow \presig(wv)) + \wt(ws_\alpha v\Rightarrow w v)
\\\underset{\text{\weightEstimateShort}}\leq&v^{-1}\mu + (-\langle \mu,\alpha\rangle +k)v^{-1}\alpha^\vee - \wt(v\Rightarrow \presig(wv)) - v^{-1}\alpha^\vee \Phi^+(w\alpha)
\\=&v^{-1}\mu + (-\langle \mu,\alpha\rangle +k-\Phi^+(w\alpha))v^{-1}\alpha^\vee - \wt(v\Rightarrow \presig(wv))
\\\leq&v^{-1}\mu-\wt(v\Rightarrow \presig(wv)).
\end{align*}
The inequality on the last line follows since $-\langle \mu,\alpha\rangle +k-\Phi^+(w\alpha)\leq 0$ (as $xa\in \Phi_\af^-$) and $v^{-1}\alpha\in \Phi^+$ by assumption.
\end{itemize}
In any case, we find an element $v\in W$ with the desired property, proving the claim for $x$.

It remains to study the case where $[b_x]>[b_{x'}]$ for all $x'<x$. By Lemma~\ref{lem:nonEmptynessBruhatCondition}, $x$ must be fundamental. The result follows from Lemma~\ref{lem:genericKottwitzFundamental}.
\end{proof}
\begin{proof}[Proof of Theorem~\ref{thm:genericGKP}]
The Theorem follows immediately from Lemmas \ref{lem:genericKottwitzLowerBound} and \ref{lem:genericKottwitzUpperBound}.
\end{proof}
\subsection{General groups}\label{sec:gnpArbitraryGroups}
In this section, we drop the assumption that $G$ should be quasi-split. We keep the notation from Section~\ref{sec:notation}. As announced, we show how to compute generic $\sigma$-conjugacy classes and classify cordial elements in this case.

The Frobenius action on the apartment $\mathcal A$ preserves the base alcove $\mathfrak a$, but no longer the chosen special vertex $\mathfrak x$. We denote by $\mu_\sigma\in V$ the uniquely determined element such that $\sigma(\mathfrak x) = \mathfrak x + \mu_\sigma$.

Moreover, there is a natural Frobenius action on $X_\ast(T)_{\Gamma_0}$. We denote the induced linear map by $\sigma_{\text{lin}}:V\rightarrow V$

Under the identification of $\mathcal A$ with $V$ by $\mathfrak x\mapsto 0$, the map $\sigma_{\text{lin}}$ is given by
\begin{align*}
\sigma_{\text{lin}}:V\rightarrow V,\quad v\mapsto \sigma(v) - \mu_\sigma.
\end{align*}
Since $\sigma_{\text{lin}}$ permutes the alcoves in $\mathcal A$, it permutes the Weyl chambers in $V$. We hence find a uniquely determined element $\sigma_1\in W$ with $\sigma_{\text{lin}}(C) = \sigma_1(C)$. Define $\sigma_2 := \sigma_1^{-1}\circ \sigma_{\text{lin}}$ such that $\sigma_2(C) = C$. Then the action of $\sigma$ on $V$ is given by the composed action
\begin{align*}
\sigma = t_{\mu_\sigma}\circ \sigma_1\circ \sigma_2,
\end{align*}
where $t_{\mu_\sigma}$ is the translation by $\mu_\sigma$. Note that $\sigma_2$ fixes both $0$ and $C$, hence it fixes $\mathfrak a$ being the only alcove in $C$ adjacent to $0$. It follows that also $t_{\mu_\sigma}\circ \sigma_1$ fixes $\mathfrak a$. So the map $t_{\mu_\sigma}\circ \sigma_1:V\rightarrow V$ \enquote{looks like} the action of an element in $\Omega\subseteq\widetilde W$, except that a lift of $\mu_\sigma\in V$ to $X_\ast(T)_{\Gamma_0}$ might not exist; and if it exists, it might not be unique.

For each $w_1, w_2\in W$, the difference $w_1\mu_\sigma-w_2\mu_\sigma$ lies in $\mathbb Z\Phi^\vee$, so we may consider $w_1\mu_\sigma-w_2 \mu_\sigma$ as a well-defined element of $X_\ast(T)_{\Gamma_0}$ even if neither $w_1\mu_\sigma$ nor $w_2 \mu_\sigma$ lies in $X_\ast(T)_{\Gamma_0}$.

We define maps \begin{align*}
\avg_{\sigma_2} :& X_\ast(T)_{\Gamma_0}\otimes\mathbb Q\rightarrow X_\ast(T)_{\Gamma_0}\otimes\mathbb Q,\\
\avg_J:& X_\ast(T)_{\Gamma_0}\otimes\mathbb Q\rightarrow X_\ast(T)_{\Gamma_0}\otimes\mathbb Q\quad (J\subseteq \Delta)
\end{align*}
as in Section~\ref{sec:parabolic-averages}. If $J = \sigma_2(J)$, we define $\pi_J := \avg_J\circ\avg_{\sigma_2}$. For an element $\mu\in X_\ast(T)_{\Gamma_0}\otimes\mathbb Q$ or $\mu\in X_\ast(T)_{\Gamma}$, we define
\begin{align*}
\conv(\mu) := \max_{\substack{J\subseteq \Delta\\
J=\sigma_2(J)}}\avg_J\avg_{\sigma_2}(\mu)\in X_\ast(T)_{\Gamma_0}\otimes\mathbb Q.
\end{align*}
%Then we can describe generic Newton points as follows:
\begin{theorem}\label{thm:gnpGeneralGroups}
%Assume that $\mathrm{char}(F)$ does not divide the order of $\pi_1(G_{\mathrm{ad}})$, the Borovoi fundamental group of the adjoint quotient\footnote{It is conjectured in \cite[Section~2.2]{Goertz2015} that this assumption can be dropped; and in fact, it does not appear any more in \cite[Section~3.2]{He2021c}.}.

Let $x = w\varepsilon^\mu\in \widetilde W$. The generic Newton point of $x$ is given by
\begin{align*}
\nu_x = \max_{v\in W} \conv\left(v^{-1}\mu - \wt(\sigma_1^{-1}v\Rightarrow\prescript{\sigma_2}{}(wv)) + \frac 1{\#W}\sum_{u\in W}(v^{-1}\mu_\sigma-u^{-1}\mu_\sigma)\right).
\end{align*}
In fact, the maximum is attained for some $v\in \LP(x)$.
\end{theorem}
We prove this theorem by reduction to the previously established results for quasi-split groups, following Goertz-He-Nie \cite[Section~2]{Goertz2015}.

By \cite[Corollary~2.2.2]{Goertz2015}, it suffices to prove the Theorem for adjoint groups, by comparing $B(G)_x$ with $B(G_{\mathrm{ad}})_x$.

Let us now assume that $G$ is adjoint. Then $\gamma := \varepsilon^{\mu_\sigma}\circ \sigma_1$ is a well-defined element of $\widetilde W$, hence of $\Omega$. Following \cite[Proposition~2.5.1]{Goertz2015}, we can identify $B(G)_x$ with $B(\tilde G)_{x\gamma}\cdot \gamma^{-1}$. Here, $\tilde G$ is a quasi-split inner form of $G$ with maximal torus $T$ and Frobenius given by $\sigma_2$. We see that
\begin{align*}
\nu_x = \nu^G\Bigl(\max_{[b]\in B(G)_x} [b]\Bigr) = \nu^G\Bigl(\max_{[b]\in B(\tilde G)_{x\gamma}} [b\gamma^{-1}]\Bigr).
\end{align*}
A quick calculation shows that for all $[b]\in B(G)$, we have
\begin{align*}
\nu^G([b]) =& \nu^{\tilde G}([b\gamma]) - \frac 1{\# W}\sum_{u\in W} u \mu_\sigma,\\
\implies \nu_x =& \nu^{\tilde G}([b_{x\gamma}]) - \frac 1{\# W}\sum_{u\in W}u\mu_\sigma.
\end{align*}
Calculating $\nu^{\tilde G}([b_{x\gamma}])$ using Corollary~\ref{cor:genericGKPMinDistance} shows Theorem~\ref{thm:gnpGeneralGroups}.

Let us return to the general situation. Following Mili\'cevi\'c-Viehmann \cite[Remark~1.3]{Milicevic2020}, we define an element $x\in \widetilde W$ to be cordial if the corresponding element $\tilde x$ in the extended affine Weyl group of the quasi-split group $\tilde G$ under the above reduction is cordial. Then the results from \cite{Milicevic2020} on cordial elements guarantee that the affine Deligne-Lusztig varieties associated with $\tilde x$ satisfy the most desirable properties as discussed earlier. By the above reduction method of \cite{Goertz2015}, it follows that also the affine Deligne-Lusztig varieties associated with $x$ satisfy these properties.
Straightforward calculation shows the following:
\begin{proposition}\label{prop:cordialGeneralGroups}
Assume that $\mathrm{char}(F)$ does not divide the order of $\pi_1(G_{\mathrm{ad}})$.

Let $x = w\varepsilon^\mu\in \widetilde W$ and pick $v\in W$ of minimal length such that
\begin{align*}
v^{-1}\mu + v^{-1}\mu_\sigma\in V
\end{align*}
is dominant. Then $\sigma_1^{-1}v\in \LP(x)$. The element $x$ is cordial if and only if the following two conditions are both satisfied:
\begin{enumerate}[(1)]
\item For any $\sigma_1^{-1}v'\in \LP(x)$, we have
\begin{align*}
d(\sigma_1^{-1}v'\Rightarrow\prescript{\sigma_2}{}(wv'))\geq 
d(\sigma_1^{-1}v\Rightarrow\prescript{\sigma_2}{}(wv)).
\end{align*}
\item We have
\begin{align*}
&d(\sigma_1^{-1}v\Rightarrow\prescript{\sigma_2}{}(wv)) = \ell\left(v^{-1}\sigma_1\prescript{\sigma_2}{}(wv)\right).\rightqed
\end{align*}
\end{enumerate}
\end{proposition}
% !TeX spellcheck = en_GB
\addcontentsline{toc}{section}{Bibliography}
\printbibliography
\end{document}